\documentclass[final,onefignum,onetabnum]{siamart171218}

\usepackage{amsmath,amsfonts}
\usepackage{graphicx}
\usepackage{subfig}
\usepackage{booktabs}
\usepackage{epstopdf}
\usepackage{algorithmic}
\usepackage[utf8]{inputenc}
\usepackage[T1]{fontenc}
\usepackage{array}
\usepackage{enumitem}
\usepackage{bm,dsfont}
\usepackage{etoolbox}
\newtoggle{arxiv}
\togglefalse{arxiv}

\ifpdf
  \DeclareGraphicsExtensions{.eps,.pdf,.png,.jpg}
\else
  \DeclareGraphicsExtensions{.eps}
\fi

\newsiamremark{remark}{Remark}
\newsiamremark{hypothesis}{Hypothesis}
\newsiamremark{example}{Example}
\crefname{hypothesis}{Hypothesis}{Hypotheses}
\newsiamthm{claim}{Claim}
\newsiamthm{assumption}{Assumption}
\crefname{assumption}{Assumption}{Assumptions}

\headers{3-Field FEM for Unsteady Implicitly Constituted Fluids}{P. E. Farrell, P. A. Gazca-Orozco, and  E. Süli}

\title{Numerical Analysis of Unsteady Implicitly Constituted
Incompressible Fluids: Three-Field Formulation\thanks{Submitted to the editors 2019.
\funding{This research is supported by the Engineering and Physical Sciences
Research Council [grant numbers EP/K030930/1, EP/R029423/1], and by the
EPSRC Centre For Doctoral Training in Partial Differential Equations:
Analysis and Applications [grant number EP/L015811/1]. The second author was partially supported by CONACyT (Scholarship 438269).}}}

\author{P. E. Farrell\thanks{Mathematical Institute, University of Oxford, UK (\email {patrick.farrell@maths.ox.ac.uk}).}
\and P. A. Gazca-Orozco\thanks{Mathematical Institute, University of Oxford, UK (\email {gazcaorozco@maths.ox.ac.uk}).}
\and E. Süli\thanks{Mathematical Institute, University of Oxford, UK (\email {endre.suli@maths.ox.ac.uk}).}}

\usepackage{amsopn}

\DeclareMathOperator*{\esslim}{ess\,lim}

\begin{document}

\maketitle
\begin{abstract}
In the classical theory of fluid mechanics a linear relationship between the shear stress and the \textcolor{black}{symmetric velocity gradient} tensor is often assumed. Even when a nonlinear relationship is assumed, it is typically formulated in terms of an explicit relation. Implicit constitutive models provide a theoretical framework that generalises this, allowing for general implicit constitutive relations. Since it is generally not possible to solve explicitly for the shear stress in the constitutive relation, a natural approach is to include the shear stress as a fundamental unknown in the formulation of the problem. In this work we present a mixed formulation with this feature, discuss its solvability and approximation using mixed finite element methods, and explore the convergence of the numerical approximations to a weak solution of the model.
\end{abstract}

\begin{keywords}
Implicitly constituted models, non--Newtonian fluids, finite element method
\end{keywords}

\begin{AMS}
65M60, 65M12, 35Q35, 76A05
\end{AMS}

\section{Implicitly constituted models}
In the classical theory of continuum mechanics the balance laws of momentum, mass, and energy do not determine completely the behaviour of a system. Additional information that captures the specific properties of the material to be studied is needed; this is what is commonly known as a \emph{constitutive relation}. The constitutive law usually expresses the stress tensor in terms of other kinematical quantities (e.g.\ the \textcolor{black}{symmetric velocity gradient}) and, even if it is nonlinear, it is typically formulated by means of an explicit relationship. It has been known for some time that in many cases explicit constitutive relations are not adequate when modeling materials with viscoelastic or inelastic responses (see e.g.\ \cite{Rajagopal:2003,Rajagopal2006}), which has led to the introduction of many ad-hoc models that try to fit the experimental data. Implicitly constituted models, introduced in \cite{Rajagopal:2003}, provide a theoretical framework that not only serves to justify these ad-hoc models, but also generalises them. The physical justification of these types of models, including a study of their thermodynamical consistency, is available and will not be discussed here; the interested reader is referred to  \cite{Rajagopal2007,Rajagopal2006,Rajagopal2008}.\\
\indent If a fluid occupies part of a space represented by a simply-connected open set $\Omega\subset \mathbb{R}^d$, where $d\in\{2,3\}$, then the evolution of the system during a given time interval $[0,T)$, for $T>0$, is determined by the usual equations of balance of mass, momentum, angular momentum and energy, which in Eulerian coordinates take the form:
\begin{align}\label{ConsLaws}
\frac{\partial \rho}{\partial t} + \text{div} (\rho \bm{u}) &= 0, \notag\\
\frac{\partial (\rho\bm{u})}{\partial t} + \text{div} (\rho \bm{u}\otimes\bm{u}) &= \text{div}\,  \bm{T} + \rho \bm{f}, \\
\bm{T} &= \bm{T}^\text{T}, \notag\\
\frac{\partial (\rho e)}{\partial t} + \text{div} (\rho e \bm{u}) &= \text{div}(\bm{T}\bm{u} - \bm{q}). \notag
\end{align}

Here:
\begin{itemize}
\item $\bm{u}:[0,T)\times\overline{\Omega} \rightarrow \mathbb{R}^d$ is the velocity field;
\item $\rho:[0,T)\times\overline{\Omega} \rightarrow \mathbb{R}$ is the density;
\item $\bm{T}:(0,T)\times\overline{\Omega} \rightarrow \mathbb{R}^{d\times d}$ is the Cauchy stress;
\item $e:[0,T)\times\overline{\Omega}\rightarrow \mathbb{R}$ is the internal energy;
\item $\bm{q}:(0,T)\times\overline{\Omega} \rightarrow \mathbb{R}^d$ is the heat flux.
\end{itemize}

The constitutive law relates the Cauchy stress (or some other appropriate measure of the stress) and the heat flux to other kinematical variables such as the shear strain, temperature, etc. In the following we will assume that the material is incompressible, \textcolor{black}{homogeneous} and undergoes an isothermal process. This implies that the energy equation decouples from the system and that the Cauchy stress can be split in two components:
\begin{equation}\label{splitting_stress}
\bm{T} = -p \bm{I} + \bm{S},
\end{equation}
\noindent
where $\bm{I}$ is the identity matrix, $p\colon(0,T)\times\Omega\rightarrow \mathbb{R}$ is the pressure (mean normal stress), and $\bm{S}\colon (0,T)\times\Omega\rightarrow \mathbb{R}^{d\times d}_{\text{sym}}$ is the shear stress (hereafter referred only as ``stress'').  In this work we will consider constitutive relations of the form
\begin{equation}\label{ConstRelG}
\bm{G}(\cdot,\bm{S},\textcolor{black}{\bm{D}(\bm{u})}) = \bm{0},
\end{equation}
\noindent
where $\bm{G}\colon Q\times \mathbb{R}_{\text{sym}}^{d\times d}\times \mathbb{R}_{\text{sym}}^{d\times d}\rightarrow \mathbb{R}\textcolor{black}{^{d\times d}_\text{sym}}$ and $\bm{D}(\bm{u}) := \frac{1}{2}(\nabla \bm{u} + (\nabla\bm{u})^\text{T})$ is the \textcolor{black}{symmetric velocity gradient}; here $Q$ is used to denote the parabolic cylinder $(0,T)\times\Omega$. The precise assumptions on this implicit function will be stated in the next section.

For a rigorous mathematical analysis of models of implicitly constituted fluids the reader is referred to \cite{Bulicek:2009,Bulicek:2012}. Existence of weak solutions for problems of this type was obtained in \cite{Bulicek:2009} and \cite{Bulicek:2012} for the steady and unsteady cases, respectively. Some extensions include \cite{Bulicek2018,Maringova2018,Prusa2018}, where additional physical responses are incorporated into the system.

As for the numerical analysis of these systems, very few results have been published so far. In \cite{Diening:2013} the convergence of a finite element discretisation to a weak solution of the problem was proved for the steady case, and the corresponding a-posteriori analysis was carried out in \cite{refId0}. More recently, this approach was extended to the time-dependent case in \cite{Sueli2018}. Also, several finite element discretisations were compared computationally in \cite{Hron2017} for problems with Bingham and stress-power-law-like rheology.

Numerical methods for the incompressible Navier--Stokes equations are usually based on a velocity-pressure formulation, and extensive studies have been carried out over the years in relation to this (see e.g.\ \cite{Girault1986,Brezzi:1991}). Such a formulation is possible, because in the case of a Newtonian fluid the explicit constitutive relation $\bm{S} = 2\mu\bm{D}(\bm{u})$ allows one to eliminate the deviatoric stress $\bm{S}$ from the momentum equation. In contrast, formulations that treat the stress as a fundamental unknown have also been introduced to study problems in elasticity and incompressible flows \cite{Arnold1984,Behr1993,Farhloul1993,Farhloul1996,Arnold2002,Ervin2008,Farhloul2008,Farhloul2009,Howell2009,Howell2016};  the key advantages of these formulations are that they are naturally applicable to nonlinear constitutive models where it is not possible to eliminate the stress, and that they allow the direct computation of the stress without resorting to numerical differentiation. In this work we will consider the mathematical analysis of a mixed formulation that treats the stress as an unknown, and illustrate its performance by means of numerical simulations.

The results here could be considered an extension of the works \cite{Diening:2013,Sueli2018,Hron2017}. One of the advantages of the approach presented here with respect to \cite{Diening:2013,Sueli2018} is that it can handle the constitutive relation in a more natural way, since the stress plays a more prominent role in the weak formulation considered. In addition, in \cite{Diening:2013,Sueli2018} no numerical simulations were presented. On the other hand, while extensive numerical computations with 3-field and 4-field formulations were performed in \cite{Hron2017}, no convergence analysis of the methods considered was discussed. The work presented here fills this gap.

\section{Preliminaries}
\subsection{Function spaces}
Throughout this work we will assume that $\Omega\subset\mathbb{R}^d$, with $d\in\{2,3\}$, is a bounded Lipschitz polygonal domain (unless otherwise stated), and use standard notation for Lebesgue, Sobolev and Bochner--Sobolev spaces (e.g.\ $(W^{k,r}(\Omega),\|\cdot\|_{W^{k,r}(\Omega)})$ and $(L^q(0,T;W^{n,r}(\Omega)),\|\cdot\|_{L^q(0,T;W^{n,r}(\Omega))})$). We will define $W^{k,r}_0(\Omega)$ for $r\in [1,\infty)$ as the closure of the space of smooth functions with compact support $C_0^\infty(\Omega)$ with respect to the norm $\|\cdot\|_{W^{k,r}(\Omega)}$ and  we will denote the dual space of $W^{1,r}_0(\Omega)$ by $W^{-1,r'}(\Omega)$. Here $r'$ is used to denote the Hölder conjugate of $r$, i.e.\ the number defined by the relation $1/r + 1/r' =1$. The duality pairing will be written in the usual way using brackets $\langle\cdot ,\cdot\rangle$. The space of traces on the boundary of functions in $W^{1,r}(\Omega)$ will be denoted by $W^{1/r',r}(\partial\Omega)$.

If $X$ is a Banach space, $C_w([0,T];X)$ will be used to denote the space of continuous functions in time with respect to the weak topology of $X$. For $r\in [1,\infty)$ we also define the following useful subspaces:
\begin{gather*}
L^r_0(\Omega) := \left\{q\in L^r(\Omega) \, : \, \int_\Omega q = 0\right\},\\
L^2_\text{div}(\Omega)^d := \overline{\{\bm{v}\in C^\infty_0(\Omega)^d\, :\, \text{div}\,\bm{v}=0\}}^{\|\cdot\|_{L^2(\Omega)}},\\
W^{1,r}_{0,\text{div}}(\Omega)^d := \overline{\{\bm{v}\in C^\infty_0(\Omega)^d\, :\, \text{div}\,\bm{v}=0\}}^{\|\cdot\|_{W^{1,r}(\Omega)}},\\
L_{\text{tr}}^r(Q)^{d\times d} := \{\bm{\tau}\in L^r(Q)^{d\times d}\, : \, \text{tr}(\bm{\tau}) = 0\},\\
L_{\text{sym}}^r(Q)^{d\times d} := \{\bm{\tau}\in L^r(Q)^{d\times d}\, : \, \bm{\tau}^\mathrm{T} = \bm{\tau}\}.
\end{gather*}

\textcolor{black}{In the definition of the space $L_{\text{tr}}^r(Q)^{d\times d}$ above, $\text{tr}(\bm{\tau})$ denotes the usual matrix trace of the $d\times d$ matrix function $\bm{\tau}$.} In the various estimates the letter $c$ will denote a generic positive constant whose exact value could change from line to line, whenever the explicit dependence on the parameters is not important.

\subsection{Interpolation inequalities}
The following embeddings will be useful when deriving various estimates. Assume that the Banach spaces $(W_1,W_2,W_3)$ form an interpolation triple in the sense that
\begin{equation*}
	\|v\|_{W_2} \leq c \|v\|_{W_1}^\lambda\|v\|_{W_3}^{1-\lambda},\quad\text{for some }\lambda\in (0,1),
\end{equation*}
and $W_1\hookrightarrow W_2 \hookrightarrow W_3$. Then (cf.\ \cite{Roubicek2013}) $L^r(0,T;W_1)\cap L^\infty(0,T;W_3)\hookrightarrow L^{r/\lambda}(0,T;W_2)$, for $r\in [1,\infty)$ and 
\begin{equation}
\|v\|_{L^{r/\lambda}(0,T;W_2)} \leq c \|v\|_{ L^\infty(0,T;W_3)}^{1-\lambda}\|v\|_{L^r(0,T;W_1)}^\lambda.
\end{equation}
An example of an interpolation triple that can be combined with this result is given by the Gagliardo--Nirenberg inequality, which states that for given $p,r\in [1,\infty)$, there is a constant $c_{p,r}>0$ such that \cite{DiBenedetto:1993}:
\begin{equation}
\|v\|_{L^s(\Omega)} \leq c_{p,r}\|\nabla v\|^\lambda_{L^r(\Omega)} \| v\|^{1-\lambda}_{L^p(\Omega)}\qquad \forall\, v\in W^{1,r}_0(\Omega)\cap L^p(\Omega),
\end{equation}

provided that $s\in[1,\infty)$ and $\lambda\in(0,1)$ satisfy
\begin{equation*}
\lambda = \frac{\frac{1}{p}-\frac{1}{s}}{\frac{1}{d} - \frac{1}{r} + \frac{1}{p}}.
\end{equation*}
A particularly useful example can be obtained if we assume that $r\textcolor{black}{>} \frac{2d}{d+2}$ and take $p=2$ and $\lambda = \frac{d}{d+2}$:
\small
\begin{equation}\label{parabolic_embedding_r}
\|v\|_{L^{\frac{r(d+2)}{d}}(Q)} \leq c \|\nabla v\|^\lambda_{L^r(Q)}\|v\|^{1-\lambda}_{L^\infty(0,T;L^2(\Omega))} \qquad \forall\, v\in L^r(0,T;W^{1,r}_0(\Omega))\cap L^\infty(0,T;L^2(\Omega)).
\end{equation}
\normalsize
\subsection{Compactness and continuity in time}
In this work we will use Simon's compactness lemma (see \cite{Simon1987}) instead of the usual Aubin--Lions lemma to extract convergent subsequences when taking the discretisation limit in the time--dependent problem. Assume that $X$ and $H$ are Banach spaces such that the compact embedding $X\hookrightarrow\hookrightarrow H$ holds. Simon's lemma states that if $\mathcal{U}\subset L^p(0,T;H)$, for some $p\in [1,\infty)$, and it satisfies:
\begin{itemize}
\item $\mathcal{U}$ is bounded in $L^1_\text{loc}(0,T;X)$;
\item $\int_0^{T-\epsilon}\|v(t+\epsilon,\cdot)-v(t,\cdot)  \|_H^p\rightarrow 0$, as $\epsilon\rightarrow 0$, uniformly for $v\in \mathcal{U}$;
\end{itemize}
then $\mathcal{U}$ is relatively compact in $L^p(0,T;H)$.

Let $X$ and $V$ be reflexive Banach spaces such that $X\hookrightarrow V$ \textcolor{black}{densely} and let $V^*$ be the dual space of $V$. The following continuity properties (see \cite{Roubicek2013}) will be important when identifying the initial condition:
\begin{gather}
v\in L^1(0,T;V^*),\, \partial_t v\in L^1(0,T;V^*)\Longrightarrow v\in C([0,T];V^*),\label{cont_time_embedding1}\\
v\in L^\infty(0,T;X)\cap C_w([0,T];V) \Longrightarrow v\in C_w([0,T];X).\label{cont_time_embedding2}
\end{gather}

\subsection{Implicit constitutive relation and its approximation}\label{implicit_c_rel}
In the mathematical analysis of these systems it is more convenient to work not with the function $\bm{G}$, but with its graph $\mathcal{A}$, which is introduced in the usual way:
\begin{equation}
(\bm{D},\bm{S})\in\mathcal{A}(\cdot) \Longleftrightarrow \bm{G}(\cdot,\bm{S},\bm{D})=\bm{0}.
\end{equation}
We will assume that $\mathcal{A}$ is a \emph{maximal monotone $r$-graph} \textcolor{black}{for some $r>1$}, which means that the following properties hold for almost every $z\in Q$:
\begin{enumerate}
\item[(A1)] [$\mathcal{A}$\textit{ includes the origin}] $(\mathbf{0},\mathbf{0})\in \mathcal{A}(z)$.
\item[(A2)] [$\mathcal{A}$\textit{ is a monotone graph}] For every $(\mathbf{D}_1,\mathbf{S}_1), (\mathbf{D}_2,\mathbf{S}_2)\in \mathcal{A}(z)$,
\[(\mathbf{S}_1 - \mathbf{S}_2):(\mathbf{D}_1 -\mathbf{D_2})\geq 0.\]
\item[(A3)] [$\mathcal{A}$\textit{ is maximal monotone}] If $(\mathbf{D},\mathbf{S}) \in\mathbb{R}_{\text{sym}}^{d\times d}\times \mathbb{R}_{\text{sym}}^{d\times d}$ is such that
\[(\hat{\mathbf{S}} - \mathbf{S}):(\hat{\mathbf{D}} -\mathbf{D})\geq 0 \quad \text{for all }(\hat{\mathbf{D}},\hat{\mathbf{S}})\in \mathcal{A}(z),\]
then $(\mathbf{D},\mathbf{S}) \in \mathcal{A}(z)$.
\item[(A4)] [$\mathcal{A}$\textit{ is an }$r$\textit{-graph}] There is a non-negative function $m\in L ^1(Q)$ and a constant $c>0$ such that
\[\mathbf{S}:\mathbf{D} \geq -m + c(|\mathbf{D}|^r + |\mathbf{S}|^{r'})\quad \text{for all } (\bm{D},\bm{S})\in\mathcal{A}(z).\]
\item[(A5)] [\textit{Measurability}] \textcolor{black}{The set-valued map $z\mapsto \mathcal{A}(z)$ is $\mathcal{L}(Q)$--$(\mathcal{B}(\mathbb{R}^{d\times d}_\text{sym}\otimes \mathbb{R}^{d\times d}_\text{sym}))$ measurable; here $\mathcal{L}(Q)$ denotes the family of Lebesgue measurable subsets of $Q$ and $\mathcal{B}(\mathbb{R}^{d\times d}_\text{sym})$ is the family of Borel subsets of $\mathbb{R}^{d\times d}_\text{sym}$.}
\item[(A6)] [\textit{Compatibility}] For any $(\bm{D},\bm{S})\in\mathcal{A}(z)$ we have that
\begin{equation*}
\text{tr}(\bm{D}) = 0 \Longleftrightarrow \text{tr}(\bm{S})=0.
\end{equation*}
\end{enumerate}
Assumption (A6) was not included in the original works \cite{Bulicek:2009,Bulicek:2012,Diening:2013}, but it is needed for consistency with the physical property that $\bm{S}$ is traceless if and only if the velocity field is divergence-free (see the discussion in \cite{Tscherpel2018}). \textcolor{black}{A very important consequence of Assumption (A5) (see \cite{Tscherpel2018}) is the existence of a measurable function (usually called a \emph{selection}) $\bm{\mathcal{D}}:Q\times \mathbb{R}^{d\times d}_{sym} \rightarrow \mathbb{R}^{d\times d}_{sym}$ such that $(\bm{\mathcal{D}}(z,\bm{\sigma}),\bm{\sigma})\in\mathcal{A}(z)$ for all $\bm{\sigma}\in \mathbb{R}^{d\times d}_{sym}$.}

In the existence results it will be useful to approximate the selection using smooth functions. To that end, let us define the mollification:
\begin{equation}\label{mollification}
\textcolor{black}{\bm{\mathcal{D}}^k(\cdot,\bm{\sigma})} := \int_{\mathbb{R}^{d\times d}_\text{sym}} \textcolor{black}{\bm{\mathcal{D}}(\cdot,\bm{\sigma} - \bm{\tau})}\rho^k(\bm{\tau})\,\text{d}\bm{\tau},
\end{equation}
where $\rho^k(\bm{\tau}) = k^{d^2}\rho(k\bm{\tau})$, $k\in\mathbb{N}$, and $\rho\in C_0^\infty(\mathbb{R}^{d\times d}_{\text{sym}})$ is a mollification kernel. It is possible to check (see e.g.\ \cite{Tscherpel2018}) that this mollification satisfies analogous monotonicity and coercivity properties to those of the selection $\textcolor{black}{\bm{\mathcal{D}}}$, i.e.\ we have that
\begin{itemize}
\item For every $\bm{\tau}_1,\bm{\tau}_2\in \mathbb{R}^{d\times d}_{\text{sym}}$ and for almost every $z\in Q$ the monotonicity condition
\begin{equation}\label{As_Sel_Aprr1}
(\textcolor{black}{\bm{\mathcal{D}}}^k(z,\bm{\tau}_1)-\textcolor{black}{\bm{\mathcal{D}}}^k(z,\bm{\tau}_2)):(\bm{\tau}_1 -\bm{\tau}_2)\geq 0
\end{equation}
holds.
\item There is a constant $C_*>0$ and a nonnegative function $g\in L^1(Q)$ such that for all $k\in\mathbb{N}$, for every $\bm{\tau}\in\mathbb{R}^{d\times d }$, and for almost every $z\in Q$ we have
\begin{equation}\label{As_Sel_Aprr2}
\bm{\tau}:\textcolor{black}{\bm{\mathcal{D}}}^k(z,\bm{\tau}) \geq -g(z) + C_*(|\bm{\tau}|^{r'} + |\textcolor{black}{\bm{\mathcal{D}}}^k(z,\bm{\tau})|^r).
\end{equation}
\item For any sequence $\{\bm{S}_k\}_{k\in\mathbb{N}}$ bounded in $L^{r'}(Q)^{d\times d}$, we have for arbitrary  $\bm{B}\in\mathbb{R}^{d\times d}_\text{sym}$ and $\phi \in C^{\infty}_0(Q)$ with $\phi \geq 0$:
\begin{equation}\label{As_Sel_Aprr3}
\liminf_{k\rightarrow\infty} \int_Q (\textcolor{black}{\bm{\mathcal{D}}}\textcolor{black}{^k}(\cdot,\bm{S}^k)-\textcolor{black}{\bm{\mathcal{D}}}(\cdot,\bm{B})):(\bm{S}^k - \bm{B})\phi(\cdot) \geq 0.
\end{equation}
\end{itemize}
It is important to remark that \eqref{As_Sel_Aprr1}, \eqref{As_Sel_Aprr2} and \eqref{As_Sel_Aprr3} are the essential properties; the explicit form \eqref{mollification} of the approximation to the selection is not very important. There are other ways to achieve the same result; for instance piecewise affine interpolation or a generalised Yosida approximation could also be used (see \cite{Sueli2018,Tscherpel2018}). The following is a localized version of Minty's lemma that will aid in the identification of the implicit constitutive relation (for a proof see \cite{Bulicek:2012aa}).
\begin{lemma}\label{LocalizedMinty}
Let $\mathcal{A}$ be a maximal monotone $r$-graph satisfying (A1)--(A4) \textcolor{black}{for some $r>1$}. Suppose that $\{\bm{D}^n\}_{n\in\mathbb{N}}$ and $\{\bm{S}^n\}_{n\in\mathbb{N}}$ are sequences of functions defined on a measurable set $\hat{Q}\subset Q$, such that:
\begin{align*}
(\bm{D}^n(\cdot),\bm{S}^n(\cdot))&\in\mathcal{A}(\cdot)\qquad\hspace{-1.5cm} &\text{a.e. in }\hat{Q},\\
\bm{D}^n &\rightharpoonup \bm{D},\qquad &\text{weakly in }L^{r}(\hat{Q})^{d\times d},\\
\bm{S}^n &\rightharpoonup \bm{S},\qquad &\text{weakly in }L^{r'}(\hat{Q})^{d\times d},\\
\limsup_{n\rightarrow\infty}\int_{\hat{Q}} \bm{S}^n&:\bm{D}^n \leq \int_{\hat{Q}} \bm{S}:\bm{D}.&
\end{align*}
Then,
\begin{equation*}
(\bm{D}(\cdot),\bm{S}(\cdot))\in \mathcal{A}(\cdot)\qquad \text{a.e. in }\hat{Q}.
\end{equation*}
\end{lemma}
The goal of this work is to prove convergence of a three-field finite element approximation of the following system:
\begin{equation}\label{PDE_goal}
  \renewcommand{\arraystretch}{1.2}
  \begin{array}{r@{}>{\null}l@{\qquad}r@{}>{\null}l}
  \partial_t\bm{u}-\text{div} (\bm{S}- \bm{u}\otimes\bm{u}) + \nabla p &= \bm{f} &  & \text{in }(0,T)\times\Omega,\\
\text{div}\, \bm{u} &= 0 &  & \text{in }(0,T)\times\Omega,\\
(\bm{D}(\bm{u}),\bm{S})  &\in \mathcal{A}(\cdot) & & \text{a.e. in }(0,T)\times\Omega,\\
\bm{u} &= \bm{0} & & \text{on }(0,T)\times\partial\Omega,\\
\bm{u}(0,\cdot) &= \bm{u}_0(\cdot) & & \text{in }\Omega,
  \end{array}  
\end{equation}
where $\mathcal{A}(\cdot)$ satisfies (A1)--(A6). The next section introduces the notation and tools that will be useful in the analysis of the discrete problem.

\subsection{Finite element approximation}
In this section, the notation and assumptions regarding the finite element approximation will be presented. Essentially the same arguments would work for any method based on a Galerkin approximation, but here we will focus only on finite element methods. Consider a family of triangulations $\{\mathcal{T}_n\}_{n\in\mathbb{N}}$ of $\Omega$ satisfying the following assumptions:
\begin{itemize}
\item (Affine equivalence). Given $n\in\mathbb{N}$ and an element $K\in \mathcal{T}_n$, there is an affine invertible mapping $\bm{F}_K\colon K\rightarrow \hat{K}$, where $\hat{K}$ is the closed standard reference simplex in $\mathbb{R}^d$.
\item (Shape-regularity). There is a constant $c_\tau$, independent of $n$, such that $$h_K \leq c_\tau \rho_K \qquad \text{for every }K\in\mathcal{T}_n,\, n\in\mathbb{N},$$
where $h_K := \text{diam}(K)$ and $\rho_K$ is the diameter of the largest inscribed ball.
\item The mesh size $h_n := \max_{K\in\mathcal{T}_n} h_K$ tends to zero as $n\rightarrow\infty$.
\end{itemize}
Define the  conforming finite element spaces associated with the triangulation $\mathcal{T}_n$:
\begin{gather*}
	V^n := \left\{\bm{v}\in \textcolor{black}{W^{1,\infty}_0(\Omega)^d}\, :\, \bm{v}|_K \circ \bm{F}_K^{-1}\in \hat{\mathbb{P}}_\mathbb{V},\, K\in\mathcal{T}_n,\, \bm{v}|_{\partial\Omega} = 0\right\},\\
	M^n := \left\{\bm{q}\in \textcolor{black}{L^\infty(\Omega)}\, :\, \bm{q}|_K \circ \bm{F}_K^{-1}\in \hat{\mathbb{P}}_\mathbb{M},\, K\in\mathcal{T}_n\right\},\\
	\Sigma^n := \left\{\bm{\sigma}\in \textcolor{black}{L^\infty(\Omega)^{d\times d}}:\, \bm{\sigma}|_K \circ \bm{F}_K^{-1}\in \hat{\mathbb{P}}_\mathbb{S},\, K\in\mathcal{T}_n\right\},
\end{gather*}
where $\hat{\mathbb{P}}_\mathbb{V}\subset W^{1,\infty}(\hat{K})^d$, $\hat{\mathbb{P}}_\mathbb{M}\subset L^\infty(\hat{K})$ and $\hat{\mathbb{P}}_\mathbb{S}\subset L^\infty(\hat{K})^{d\times d}$ are finite-dimensional polynomial subspaces on the reference simplex $\hat{K}$. Each of these spaces will be assumed to have a finite and locally supported basis. As in the continuous case, it will be useful to introduce the following finite-dimensional subspaces \textcolor{black}{for $r>1$}:
\begin{gather*}
M^n_0 := M^n \cap L_0^{r'}(\Omega),\quad \Sigma^n_\text{tr} := \Sigma^n\cap L^{r}_\text{tr}(\Omega)^{d\times d},\quad \Sigma^n_\text{sym} := \Sigma^n\cap L^{r}_\text{sym}(\Omega)^{d\times d},\\
V^n_\text{div} := \left\{\bm{v}\in V^n\, :\, \int_\Omega q\,\text{div}\bm{v} = 0,\quad \forall\, q\in M^n\right\},\\
\Sigma^n_\text{div}(\bm{f}) := \left\{\bm{\sigma}\in \Sigma^n_\text{sym} \, :\,  \int_\Omega \bm{\sigma}:\bm{D}(\bm{v}) = \langle \bm{f},\bm{v}\rangle,\quad \forall\, \bm{v}\in V^n_\text{div}\right\}.
\end{gather*}
\begin{assumption}[Approximability]\label{Approximability}
For every $s\in [1,\infty)$ we have that 
\begin{gather*}
\inf_{\overline{\bm{v}}\in V^n}\|\bm{v}-\overline{\bm{v}}\|_{W^{1,s}(\Omega)} \rightarrow 0 \quad \text{ as }n\rightarrow\infty\quad \forall\,\bm{v}\in W^{1,s}_0(\Omega)^d,\\
\inf_{\overline{q}\in M^n}\|q-\overline{q}\|_{L^{s}(\Omega)} \rightarrow 0 \quad \text{ as }n\rightarrow\infty\quad \forall\,q\in L^s(\Omega),\\
\inf_{\overline{\bm{\sigma}}\in \Sigma^n}\|\bm{\sigma}-\overline{\bm{\sigma}}\|_{L^{s}(\Omega)} \rightarrow 0 \quad \text{ as }n\rightarrow\infty\quad \forall\,\bm{\sigma}\in L^s(\Omega)^{d\times d}. 
\end{gather*}
\end{assumption}
\begin{assumption}[Projector $\Pi^n_\Sigma$]\label{Projector_Stress}
For each $n\in\mathbb{N}$ there is a linear projector $\Pi^n_\Sigma :L_\text{\emph{sym}}^1(\Omega)^{d\times d}\rightarrow \Sigma^n_\text{\emph{sym}}$ such that:
\begin{itemize}[leftmargin = 0.8cm]
\item (Preservation of divergence). For every $\bm{\sigma}\in L^{1}(\Omega)^{d\times d}$ we have 
\begin{equation*}
\int_\Omega \bm{\sigma}:\bm{D}(\bm{v}) = \int_\Omega \Pi^n_\Sigma(\bm{\sigma}):\bm{D}(\bm{v}) \quad \forall\, \bm{v}\in V^n_\text{\emph{div}}.
\end{equation*}
\item ($L^{s}$--stability). For every $s\in (1,\infty)$ there is a constant $c>0$, independent of $n$, such that:
\begin{equation*}
\|\Pi^n_\Sigma \bm{\sigma}\|_{L^{s}(\Omega)} \leq c \|\bm{\sigma}\|_{L^{s}(\Omega)}\qquad \forall\, \bm{\sigma}\in L_\text{\emph{sym}}^s(\Omega)^{d\times d}.
\end{equation*}
\end{itemize}
\end{assumption}
\begin{assumption}[Projector $\Pi^n_V$]\label{Projector_Velocity}
For each $n\in\mathbb{N}$ there is a linear projector $\Pi^n_V :W^{1,1}_0(\Omega)^d\rightarrow V^n$ such that the following properties hold:
\begin{itemize}[leftmargin = 0.8cm]
\item (Preservation of divergence). For every $\bm{v}\in W^{1,1}_0(\Omega)^d$ we have 
\begin{equation*}
\int_\Omega q\,\text{\emph{div}}\,\bm{v} = \int_\Omega q\,\text{\emph{div}}(\Pi^n_V\bm{v}) \quad\, \forall\, q\in M^n.
\end{equation*}
\item ($W^{1,s}$--stability). For every $s\in (1,\infty)$ there is a constant $c>0$, independent of $n$, such that:
\begin{equation*}
\|\Pi^n_V\bm{v}\|_{W^{1,s}(\Omega)} \leq c \|\bm{v}\|_{W^{1,s}(\Omega)}\qquad \forall \,\bm{v}\in W^{1,s}_0(\Omega)^d.
\end{equation*}
\end{itemize}
\end{assumption}
\begin{assumption}[Projector $\Pi^n_M$]\label{Projector_Pressure}
For each $n\in\mathbb{N}$ there is a linear projector $\Pi^n_M :L^1(\Omega)\rightarrow M^n$ such that for all $s\in (1,\infty)$ there is a constant $c>0$, independent of $n$, such that:
\begin{equation*}
\|\Pi^n_M q\|_{L^{s}(\Omega)} \leq c \|q\|_{L^{s}(\Omega)}\qquad \forall\, q\in L^s(\Omega).
\end{equation*}
\end{assumption}

\noindent
It is not difficult to show that the approximability and stability properties imply that for $s\in [1,\infty)$ we have:
\begin{gather}
\|\bm{\sigma} - \Pi^n_\Sigma\bm{\sigma}\|_{L^{s}(\Omega)} \rightarrow 0\quad \text{ as }n\rightarrow\infty\quad \forall\, \bm{\sigma}\in L_\text{sym}^s(\Omega)^{d\times d},\notag\\
\|\bm{v} - \Pi^n_V\bm{v}\|_{W^{1,s}(\Omega)} \rightarrow 0\quad \text{ as }n\rightarrow\infty\quad \forall \,\bm{v}\in W^{1,s}(\Omega)^{d},\label{Approx_FEMconv}\\
\|q - \Pi^n_Mq\|_{L^{s}(\Omega)} \rightarrow 0\quad \text{ as }n\rightarrow\infty\quad \forall\, q\in L^{s}(\Omega).\notag
\end{gather}
\begin{remark}
A very important consequence of the previous assumptions is the existence, for every $s\in (1,\infty)$, of two positive constants $\beta_s,\gamma_s >0$, independent of $n$, such that the following discrete inf-sup conditions hold:
\begin{gather}
\inf_{q\in M^n_0}\hspace{0.18cm}\sup_{\bm{v}\in V^n} \frac{\int_\Omega q\,\text{div}\,\bm{v}}{\|\bm{v}\|_{W^{1,s}(\Omega)}\|q\|_{L^{s'}(\Omega)}} \geq \beta_s,\label{infsupVel_A}\\
\hspace{0.12cm}\inf_{\bm{v}\in V^n_{\text{div}}}\sup_{\bm{\tau}\in \Sigma^n_\text{sym}}  
\frac{\int_\Omega \bm{\tau}:\bm{D}(\bm{v})}{\|\bm{\tau}\|_{L^{s'}(\Omega)}\|\bm{v}\|_{W^{1,s}(\Omega)}} \geq \gamma_s.\label{infsupStress_A}
\end{gather}
\end{remark}

\begin{example}
There are several pairs of velocity-pressure spaces known to satisfy the stability \cref{Approximability,Projector_Velocity}. They include the conforming Crouzeix--Raviart element, the MINI element, the $\mathbb{P}_2$--$\mathbb{P}_0$ element and the Taylor--Hood element $\mathbb{P}_k$--$\mathbb{P}_{k-1}$ for $k\geq d$ (see \cite{Belenki:2012,Boffi2013,Diening:2013,Girault2003,Crouzeix1973}). In addition to stability, the Scott--Vogelius element also satisfies the property that the discretely divergence-free velocities are pointwise divergence-free (the stability can be guaranteed by assuming for example that the mesh has been barycentrically refined, see \cite{Scott1985}); another example of a velocity-pressure pair with this property is given by the Guzm\'{a}n--Neilan element \cite{Guzman2014,Guzman2014a}. To satisfy Assumption \ref{Projector_Pressure}, one could use the Clément interpolant \cite{Clement1975}.

Sometimes it is easier to prove the inf-sup condition directly. For example, if the space of discrete stresses consists of discontinuous $\mathbb{P}_k$ polynomials (with $k\geq 1$):
\begin{equation*}
\Sigma^n = \{\bm{\sigma}\in L^\infty(\Omega)^{d\times d}\, :\, \bm{\sigma}|_K\in\mathbb{P}_k(K)^{d\times d},\text{ for all }K\in \mathcal{T}_n\},
\end{equation*}
and we have that $\bm{D}(V^n) \subset \Sigma^n$ (e.g.\ we could take the Taylor--Hood element $\mathbb{P}_{k+1}$--$\mathbb{P}_k$ for the velocity and the pressure), then the inf-sup condition follows from the fact that for $s\in (1,\infty)$ there is a constant $c>0$, independent of $h$, such that for any $\bm{\sigma}\in \Sigma^n$ there is $\bm{\tau}\in\Sigma^n$ such that \cite{Sandri1998}:
\begin{equation*}
\int_\Omega \bm{\tau}:\bm{\sigma} = \|\bm{\sigma}\|_{L^s(\Omega)}^s\quad\text{ and   }\quad \|\bm{\tau}\|_{L^{s'}(\Omega)}\leq c\|\bm{\sigma}\|_{L^s(\Omega)}^{s-1}.
\end{equation*}
In case a continuous piecewise polynomial approximation of the stress is preferred, one could use the conforming Crouzeix--Raviart element for the discrete  velocity and pressure and the following space for the stress \cite{Ruas1994} :
\begin{equation*}
\Sigma^n = \{\bm{\sigma}\in C(\overline{\Omega})^{d\times d}\, : \,  \bm{\sigma}|_K\in(\mathbb{P}_1(K)\oplus \mathcal{B})^{d\times d},\text{ for all }K\in \mathcal{T}_n\},
\end{equation*}
where
\begin{equation*}
\mathcal{B} := \text{span}\, \{\lambda_1^2\lambda_2\lambda_3,\lambda_1\lambda_2^2\lambda_3,\lambda_1\lambda_2\lambda_3^2\},
\end{equation*}
and $\{\lambda_j\}_{j=1}^3$ are barycentric coordinates on $K$.
\end{example}
\begin{remark}\label{remark_divFree}
If the discretely divergence-free velocities are in fact exactly divergence free, i.e.\ if $V^n_\text{div}\subset W^{1,r}_{0,\text{div}}(\Omega)^d$, and $\bm{D}(V^n)\subset \Sigma^n$, then the stress-velocity inf-sup condition also holds for the subspace of traceless stresses. Consequently, fewer degrees of freedom are needed to compute the stress unknowns.
\end{remark}
\subsection{Time discretisation}
In this section we will describe the notation that will be used when performing the time discretisation of the problem. Let $\{\tau_m\}_{m\in\mathbb{N}}$ be a sequence of time steps such that $T/\tau_m\in \mathbb{N}$ and $\tau_m \rightarrow 0$, as $m\rightarrow\infty$. For each $m\in \mathbb{N}$ we define the equidistant grid:
\begin{equation*}
\{t^m_j\}_{j=0}^{T/\tau_m},\qquad t_j = t_j^m := j\tau_m.
\end{equation*}
This can be used to define the parabolic cylinders $Q_{i}^j := (t_i,t_j)\times \Omega$, where $0\leq i\leq j \leq T/\tau_m$.
Also, given a set of functions $\{v^j\}_{j=0}^{T/\tau_m}$ belonging to a Banach space $X$, we can define the piecewise constant interpolant $\overline{v}\in L^\infty(0,T;X)$ as:
\begin{equation}\label{pwise_constant_interp}
\overline{v}(t) := v^j,\qquad t\in (t_{j-1},t_j],\qquad j\in \{1,\ldots,T/\tau_m\},
\end{equation}
and the piecewise linear interpolant $\tilde{v}\in C([0,T];X)$ as:
\begin{equation}\label{pwise_linear_interp}
\tilde{v}(t):= \frac{t - t_{j-1}}{\tau_m}v^j + \frac{t_j -t }{\tau_m}v^{j-1},\qquad t \in [t_{j-1},t_j],\quad j \in \{1,\ldots,T/\tau_m\}.
\end{equation}
For a given function $g\in L^p(0,T;X)$, with $p\in [1,\infty)$, we define the  time averages:
\begin{equation}\label{time_averages}
g_j(\cdot) := \frac{1}{\tau_m}\int_{t_{j-1}}^{t_j}g(t,\cdot)\,\text{d}\,t,\quad j\in\{1,\ldots, T/\tau_m\}.
\end{equation}
Then the piecewise constant interpolant $\overline{g}$ defined by \eqref{pwise_constant_interp} satisfies \cite{Roubicek2013}:
\begin{equation}
\|\overline{g}\|_{L^p(0,T;X)} \leq \|g\|_{L^p(0,T;X)},
\end{equation}
and
\begin{equation}\label{conv_time_interpolant}
\overline{g} \rightarrow g \text{ strongly in } L^p(0,T;X), \text{ as }m\rightarrow\infty.
\end{equation}
\section{Weak formulation}
In this section we will present a weak formulation for the problem \eqref{PDE_goal}, where now we assume that $\bm{f}\in L^{r'}(0,T;W^{-1,r'}(\Omega)^d)$, $\bm{u}_0\in L_\text{div}^2(\Omega)^d$ and the graph $\mathcal{A}$ satisfies the assumptions (A1)--(A6) for some $r>\frac{2d}{d+2}$.
Similarly to previous works on the analysis of implicitly constituted fluids, a Lipschitz truncation technique will be required when proving that the limit of the sequence of approximate solutions satisfies the constitutive relation. The theory of Lipschitz truncation for time-dependent problems is not as well developed as in the steady case; here it will be necessary to work locally and the equation plays a vital role (several versions of parabolic Lipschitz truncation have appeared in the literature, see e.g.\ \cite{Diening2010,Bulicek:2012,BREIT2013,Diening2017}). Since the pressure will not be present in the weak formulation, it will be more convenient to use the construction developed in \cite{BREIT2013} because it preserves the solenoidality of the velocity. The following lemma states the main properties of this solenoidal Lipschitz truncation.

\begin{lemma}{\emph{(\cite{BREIT2013,Sueli2018})}}\label{LipschitzTr_unsteady}
Let $p\in (1,\infty)$, $\sigma\in (1,\min(p,p'))$ and let $Q_0 = I_0\times B_0\subset \mathbb{R}\times\mathbb{R}^3$ be a parabolic cylinder, where $I_0$ is an open interval and $B_0$ is an open ball. \textcolor{black}{Denote by $\alpha Q_0$, where $\alpha>0$, the $\alpha$-scaled version of $Q_0$ keeping the barycenter the same.} Suppose $\{\bm{e}^l\}_{l\in\mathbb{N}}$ is a sequence of divergence-free functions that is uniformly bounded in $L^\infty(I_0;L^\sigma(\textcolor{black}{B_0})^d)$ and converges to zero weakly in $L^p(I_0;W^{1,p}(B_0)^d)$ and strongly in $L^\sigma(Q_0)^d$. Let $\{\bm{G}_1^l\}_{l\in\mathbb{N}}$ and $\{\bm{G}_2^l\}_{l\in\mathbb{N}}$ be sequences  that converge to zero weakly in $L^{p'}(Q_0)^{d\times d}$ and strongly in $L^{\sigma}(Q_0)^{d\times d}$, respectively. Define $\bm{G}^l := \bm{G}^l_1 + \bm{G}^l_2$ and suppose that, for any $l\in\mathbb{N}$, the equation
\begin{equation}\label{eq_error_LipTr}
\int_{Q_0} \partial_t\bm{e}^l\cdot \bm{w} = \int_{Q_0} \bm{G}^l:\nabla \bm{w}\quad \forall\, \bm{w}\in C^\infty_{0,\text{\emph{div}}}(Q_0)^d. 
\end{equation}
is satisfied. Then there is a number $j_0\in\mathbb{N}$, a sequence $\{\lambda_{l,j}\}_{l,j\in\mathbb{N}}$ with $2^{2^j}\leq \lambda_{l,j}\leq 2^{2^{j+1}-1}$, a sequence of functions $\{\bm{e}^{l,j}\}_{l,j\in\mathbb{N}}\subset L^1(Q_0)^d$, a sequence of open sets $\mathcal{B}_{\lambda_{l,j}}\subset Q_0$, for $l,j\in \mathbb{N}$, and a function $\zeta\in C^\infty_0(\frac{1}{6}Q_0)$ with $\mathds{1}_{\frac{1}{8}Q_0} \leq \zeta\leq \mathds{1}_{\frac{1}{6}Q_0}$ with the following properties:
\begin{enumerate}[leftmargin = 0.8cm]
\item $\bm{e}^{l,j}\in L^q(\frac{1}{4}I_0;W^{1,q}_{0,\text{\emph{div}}}(\frac{1}{6}B_0)^d)$ for any $q\in [1,\infty)$ and $\text{\emph{supp}}(\bm{e}^{l,j})\subset \frac{1}{6}Q_0$, for any $j\geq j_0$ and any $l\in\mathbb{N}$;
\item $\bm{e}^{l,j} = \bm{e}^j$ on $\frac{1}{8}Q_0\setminus\mathcal{B}_{\lambda_{l,j}}$, for any $j\geq j_0$ and any $l\in\mathbb{N}$;
\item There is a constant $c>0$ such that 
\begin{equation*}
\limsup_{l\rightarrow\infty}\lambda^p_{l,j} |\mathcal{B}_{\lambda_{l,j}}|\leq c 2^{-j},\quad \text{for any }j\geq j_0;
\end{equation*}
\item For $j\geq j_0$ fixed, we have as $l\rightarrow\infty$:
\begin{align*}
\bm{e}^{l,j} &\rightarrow \bm{0}, \quad &\text{strongly in } L^{\infty}(\tfrac{1}{4}Q_0)^{d},\notag\\
\nabla\bm{e}^{l,j} &\rightharpoonup \bm{0}, \quad &\text{weakly in } L^{q}(\tfrac{1}{4}Q_0)^{d\times d},\quad\forall\, q\in[1,\infty);\notag\\
\end{align*}
\item There is a constant $c>0$ such that:
\begin{equation*}
\limsup_{l\rightarrow\infty}\left|\int_{Q_0} \bm{G}^l:\nabla\bm{e}^{l,j} \right| \leq c 2^{-j},\quad \text{for any }j\geq j_0;
\end{equation*}
\item There is a constant $c>0$ such that for any $\bm{H}\in L^{p'}(\frac{1}{6}Q_0)^{d\times d}$:
\begin{equation*}
\limsup_{l\rightarrow\infty}\left|\int_{Q_0} (\bm{G}_1^l + \bm{H}):\nabla\bm{e}^{l,j}\zeta \mathds{1}_{\mathcal{B}^c_{\lambda_{l,j}}} \right| \leq c 2^{-j/p},\quad \text{for any }j\geq j_0.
\end{equation*}
\end{enumerate}
\end{lemma}

\subsection{Mixed formulation and time--space discretisation}
Before we present the weak formulation, let us define
\begin{equation*}
\check{r} := \min\left\{\frac{r(d+2)}{2d},r' \right\}.
\end{equation*}
The weak formulation for \eqref{PDE_goal} then reads as follows.

\textbf{Formulation $\check{\text{A}}$.} Find functions
\begin{equation*}
\begin{gathered}
\bm{S}\in L^{r'}_\text{sym}(Q)^{d\times d}\cap  L^{r'}_\text{tr}(Q)^{d\times d}, \\
\bm{u}\in L^r(0,T;W^{1,r}_{0,\text{div}}(\Omega)^d)\cap L^\infty(0,T;L^2_\text{div}(\Omega)^d),\\
\partial_t \bm{u} \in L^{\check{r}}(0,T; (W_{0,\text{div}}^{1,\check{r}'}(\Omega)^d)^*),
\end{gathered}
\end{equation*}
such that
\begin{gather*}
\langle\partial_t \bm{u},\bm{v} \rangle + \int_\Omega (\bm{S} - \bm{u}\otimes\bm{u}):\bm{D}(\bm{v}) = \langle \bm{f},\bm{v}\rangle \qquad \hspace{-0.1cm}\forall\, \bm{v}\in W^{1,\check{r}'}_{0,\text{div}}(\Omega)^d, \text{ a.e. }t\in (0,T),\\
\textcolor{black}{(\bm{D}(\bm{u}),\bm{S})\in \mathcal{A}(\cdot), \text{ a.e. in }  (0,T)\times\Omega,}\\
\esslim_{t\rightarrow 0^+} \|\bm{u}(t,\cdot) - \bm{u}_0(\cdot) \|_{L^2(\Omega)} = 0.
\end{gather*}
\begin{remark}\label{remark_nopressure}
In the formulation above all the test-velocities are divergence-free and as a consequence the presure term vanishes. In this section we will carry out the analysis for the velocity and stress variables only. It is known that even in the Newtonian case (i.e.\ $r=2$) the pressure is only a distribution in time, when working with a no-slip boundary condition (see e.g.\ \cite{Galdi2011}). An integrable pressure can be obtained if Navier's slip boundary condition is used instead \cite{Bulicek:2012}, but in this work we will confine ourselves to the more common no-slip boundary condition.
\end{remark}
\begin{remark}
From \eqref{cont_time_embedding1} we have that 
\begin{equation*}
\bm{u}\in C([0,T];(W_{0,\text{div}}^{1,\check{r}'}(\Omega)^d)^*)\hookrightarrow C_w([0,T];(W_{0,\text{div}}^{1,\check{r}'}(\Omega)^d)^*),
\end{equation*}
and since $\check{r}\leq r'$ we also know that \textcolor{black}{$L^2_\text{div}(\Omega)^d \hookrightarrow (W_{0,\text{div}}^{1,\check{r}'}(\Omega)^d)^*$}. Combined with \eqref{cont_time_embedding2} this yields $\bm{u}\in C_w([0,T];L^2_\text{div}(\Omega)^d)$ and hence the initial condition only makes sense a priori in this weaker sense. However, for this problem it will be proved that it also holds in the stronger sense described above.
\end{remark}

For a given time step $\tau_m$ and $j\in\{1,\ldots,T/\tau_m\}$, let $\bm{f}_j\in W^{-1,r'}(\Omega)^d$ and $\textcolor{black}{\bm{\mathcal{D}}}^k_j:\Omega\times\mathbb{R}^{d\times d}\rightarrow \mathbb{R}^{d\times d}$ be the time averages associated with $\bm{f}$ and $\textcolor{black}{\bm{\mathcal{D}}}^k$, respectively (recall \eqref{time_averages}). The time derivative will be discretised using an implicit Euler scheme; higher order time stepping techniques might not be more advantageous here because higher regularity in time of weak solutions to the problem is not guaranteed a priori. The discrete formulation of the problem can now be introduced.\\

\textbf{Formulation $\check{\text{A}}_{\text{k},\text{n},\text{m},\text{l}}$.} For $j\in\{1,\ldots,T/\tau_m\}$, find functions $\bm{S}_j^{k,n,m,l}\in \Sigma_\text{sym}^n$ and $\bm{u}_j^{k,n,m,l}\in V_\text{div}^n$ such that:
\small
\begin{align*}
\int_\Omega (\textcolor{black}{\bm{\mathcal{D}}}_j^k(\cdot,\bm{S}_j^{k,n,m,l}) - \bm{D}(\bm{u}_j^{k,n,m,l})):\bm{\tau} &= 0 &\forall\, \bm{\tau}\in \Sigma^n_{\text{sym}},\\
\frac{1}{\tau_m}\int_\Omega (\bm{u}_j^{k,n,m,l} - \bm{u}_{j-1}^{k,n,m,l})\cdot \bm{v}  + \frac{1}{l}\int_\Omega |\bm{u}_j^{k,n,m,l}|^{2r'-2}&\bm{u}_j^{k,n,m,l}\cdot\bm{v} & \\ +\int_\Omega   (\bm{S}_j^{k,n,m,l}:\bm{D}(\bm{v})  + \mathcal{B}(\bm{u}_j^{k,n,m,l}&,\bm{u}_j^{k,n,m,l},\bm{v}))  = \langle \bm{f}_j,\bm{v}\rangle &\forall\, \bm{v}\in V^n_\text{div},\\
\bm{u}^{k,n,m,l}_0 = P^n_\text{div}\bm{u}_0. & &
\end{align*}
\normalsize
Here $P^n_\text{div}:L^2(\Omega)^d\rightarrow V^n_\text{div}$ is simply the $L^2$--projection defined through
\begin{equation}
\int_\Omega P^n_\text{div}\bm{v}\cdot\bm{w} = \int_\Omega \bm{v}\cdot\bm{w}\qquad \forall\, \bm{w}\in V^n_\text{div}.
\end{equation}
The form $\mathcal{B}$ is meant to represent the convective term and \textcolor{black}{is defined for functions $\bm{u},\bm{v},\bm{w}\in C^\infty_0(\Omega)^d$ as}:
\begin{equation*}
\renewcommand{\arraystretch}{1.5}
\mathcal{B}(\bm{u},\bm{v},\bm{w}) := \left\{
\begin{array}{cc}
- \displaystyle\int_\Omega\bm{u}\otimes\bm{v}:\bm{D}(\bm{w}), & \textrm{ if } V^n_\textrm{div} \subset W^{1,r}_{0,\textrm{div}}(\Omega)^d,\\
  \displaystyle\frac{1}{2}\int_\Omega  \bm{u}\otimes\bm{w}:\bm{D}(\bm{v})-\bm{u}\otimes\bm{v}:\bm{D}(\bm{w}), & \textrm{ otherwise}.\\
\end{array}
\right.
\end{equation*}
This definition guarantees that $\mathcal{B}(\bm{v},\bm{v},\bm{v})=0$ for every $\bm{v}$ for which this expression is well defined, regardless of whether $\bm{v}$ is pointwise divergence-free or not, which is very useful when obtaining a priori estimates; it reduces to the usual weak form of the convective term whenever the velocities are exactly divergence-free. It is now necessary to check that $\mathcal{B}$ \textcolor{black}{can be continuously extended to} the spaces involving time. By standard function space interpolation, we have that for almost every $t\in(0,T)$:
\small
\begin{align*}
\int_\Omega |\bm{u}(t,\cdot)&\otimes \bm{v}(t,\cdot):\bm{D}(\bm{w}(t,\cdot))| \leq \|\bm{u}(t,\cdot)\|_{L^{2\check{r}}(\Omega)} \| \bm{v}(t,\cdot)\|_{L^{2\check{r}}(\Omega)} \|\bm{D}(\bm{w}(t,\cdot))\|_{L^{\check{r}'}(\Omega)} \\
&\leq \|\bm{u}(t,\cdot)\|_{L^{\frac{r(d+2)}{d}}(\Omega)}\|\bm{v}(t,\cdot)\|_{L^{\frac{r(d+2)}{d}}(\Omega)}\|\bm{D}(\bm{w}(t,\cdot))\|_{L^{\check{r}'}(\Omega)}\\
&\leq c \|\bm{u}(t,\cdot)\|_{W^{1,r}(\Omega)} \|\bm{v}(t,\cdot)\|_{W^{1,r}(\Omega)} \|\bm{w}(t,\cdot)\|_{W^{1,\check{r}'}(\Omega)}.
\end{align*}
\normalsize
As in the steady case (cf.\ \cite{Diening:2013}), a more restrictive condition is needed in order to bound the additional term in $\mathcal{B}$ whenever the elements are not exactly divergence-free. Namely, if we assume that $r\geq \frac{2(d+1)}{d+2}$ (this is the analogue of the condition $r\geq \frac{2d}{d+1}$ in the steady case) then there is a $q\in (1,\infty]$ such that $\frac{1}{r} + \frac{d}{r(d+2)} + \frac{1}{q} = 1$, and therefore
\small
\begin{align*}
\int_\Omega |\bm{u}(t,\cdot)&\otimes \bm{w}(t,\cdot):\bm{D}(\bm{v}(t,\cdot))| \leq \|\bm{u}(t,\cdot)\|_{L^{\frac{r(d+2)}{d}}(\Omega)}\|\bm{D}(\bm{v}(t,\cdot))\|_{L^{r}(\Omega)}\|\bm{w}(t,\cdot)\|_{L^{q}(\Omega)}\\
&\leq c\|\bm{u}(t,\cdot)\|_{W^{1,r}(\Omega)} \|\bm{v}(t,\cdot)\|_{W^{1,r}(\Omega)} \|\bm{w}(t,\cdot)\|_{W^{1,\check{r}'}(\Omega)}.
\end{align*}
\normalsize
\textcolor{black}{On the other hand, using Hölder's inequality we can also obtain the estimate}
\begin{align*}
\textcolor{black}{\|\mathcal{B}(\bm{u},\bm{v},\bm{w})\|_{L^1(0,T)} }&\textcolor{black}{\leq \|\bm{u}\|_{L^{2r'}(Q)} \|\bm{v}\|_{L^{2r'}(Q)}\|\bm{w}\|_{L^r(0,T;W^{1,r}(\Omega))}}\\
	&\textcolor{black}{+ \|\bm{u}\|_{L^{2r'}(Q)} \|\bm{w}\|_{L^{2r'}(Q)}\|\bm{v}\|_{L^r(0,T;W^{1,r}(\Omega))},}
\end{align*}
\textcolor{black}{which means that if the $L^{2r'}(Q)^d$ norm of $\bm{u}$ is finite, then the additional restriction $r\geq \frac{2(d+1)}{d+2}$ is not needed. Moreover, this would also imply that the velocity is an admissible test function, which is useful in the convergence analysis. This motivates the introduction of the penalty term in Formulation $\check{\text{A}}_{\text{k},\text{n},\text{m},\text{l}}$. }
\begin{remark}
While Formulation $\check{\text{A}}_{\text{k},\text{n},\text{m},\text{l}}$ does not contain the pressure, in practice the incompressibility condition is enforced through the addition of a Lagrange multiplier $p^{k,n,m,l}_j\in M_0^n$, which could be thought of as the pressure in the system (the reason for the omission of the pressure in the analysis is explained in \cref{remark_nopressure}). For this reason it is necessary to consider additional assumptions that guarantee inf-sup stability of the spaces $V^n$ and $M^n$ (see \cref{Projector_Velocity,Projector_Pressure}). In case the problem does have an integrable pressure $p$, then it is expected that the sequence of discrete pressures converges to it in $L^1(Q)$.
\end{remark}
\begin{remark}
	\textcolor{black}{Assumption (A5) also implies the existence of a selection $\bm{\mathcal{S}}:Q\times \mathbb{R}^{d\times d}_{sym} \rightarrow \mathbb{R}^{d\times d}_{sym}$ such that $(\bm{\tau},\bm{\mathcal{S}}(z,\bm{\tau}))\in\mathcal{A}(z)$ for all $\bm{\tau}\in \mathbb{R}^{d\times d}_{sym}$, and some models can be written more naturally with a selection of this form; the same analysis \textcolor{black}{as the one} presented in this work can be applied to that situation.}
In fact, in practice it is not necessary to find a selection in order to perform the computations, i.e.\ in the simulations it is possible to work directly with the implicit function $\bm{G}$. When performing the analysis though, the function $\bm{G}$ is not appropriate because many different expressions could lead to the same constitutive relation, but have different mathematical properties.
\end{remark}
\begin{remark}
In this work we did not consider a dual formulation, e.g.\ based on $H(\text{div};\Omega)$, because for the unsteady problem we do not have at our disposal results that guarantee the integrability of $\text{div}\,\bm{S}$. 
\end{remark}
In the next theorem, convergence of the sequence of discrete solutions to a weak solution of the problem is proved. Since the ideas and arguments contained in the proof are similar to the ones presented in the previous sections and follow a similar approach to \cite{Sueli2018}, we will not include here all the details of the calculations unless there is a significant difference.

\begin{theorem}\label{Convergence_unsteady_FormA} Assume that $r>\frac{2d}{d+2}$, let $\{\Sigma^n,V^n,M^n\}_{n\in\mathbb{N}}$ be a family of finite element spaces satisfying Assumptions \ref{Approximability}--\ref{Projector_Velocity}. \textcolor{black}{Then for $k,n,m,l\in \mathbb{N}$ there exists a sequence $\{(\bm{S}^{k,n,m,l}_j,\bm{u}^{k,n,m,l}_j)\}_{j=1}^{T/\tau_m}$ of solutions of Formulation $\check{\text{\emph{A}}}_{\text{\emph{k}},\text{\emph{n}},\text{\emph{m}},\text{\emph{l}}}$, and a couple 
$(\bm{S},\bm{u})\in L^{r'}_\text{\emph{sym}}(Q)^{d\times d}\cap L^{r'}_\text{\emph{tr}}(Q)^{d\times d}\times L^r(0,T;W^{1,r}_{0,\text{\emph{div}}}(\Omega)^d)\cap L^\infty(0,T;L^2_\text{\emph{div}}(\Omega)^d)$} such that the corresponding time interpolants (recall \eqref{pwise_constant_interp} and \eqref{pwise_linear_interp}) $\overline{\bm{u}}^{k,n,m,l}$, $\tilde{\bm{u}}^{k,n,m,l}$ and $\overline{\bm{S}}^{k,n,m,l}$ satisfy (up to a subsequence):
\begin{align}\label{WEakUnstFormA}
\overline{\bm{S}}^{k,n,m,l} &\rightharpoonup \bm{S} \quad &\text{weakly in } L^{r'}(Q)^{d\times d},\notag\\
\overline{\bm{u}}^{k,n,m,l} &\rightharpoonup \bm{u} \quad &\text{weakly in } L^r(0,T;W^{1,r}_0(\Omega)^d),\\
\overline{\bm{u}}^{k,n,m,l},\tilde{\bm{u}}^{k,n,m,l} &\overset{\ast}{\rightharpoonup} \bm{u} \quad &\text{weakly* in } L^\infty(0,T;L^{2}(\Omega)^d),\notag
\end{align} 
and $(\bm{S},\bm{u})$ solves Formulation $\check{\text{\emph{A}}}$, with the limits taken in the order $k\rightarrow \infty$, $(n,m)\rightarrow\infty$ and $l\rightarrow\infty$.
\end{theorem}
\begin{proof}
The idea of the proof is common in the analysis of nonlinear PDE: we obtain a priori estimates and use compactness arguments to pass to the limit in the equation. In order to prove the existence of solutions of Formulation $\check{\text{A}}_{\text{k},\text{n},\text{m},\text{l}}$, we need to check that given $(\bm{S}^{k,n,m,l}_{j-1},\bm{u}^{k,n,m,l}_{j-1})$, we can find $(\bm{S}^{k,n,m,l}_{j},\bm{u}^{k,n,m,l}_{j})$, for $j\in\{1,\ldots,T/\tau_m\}$. Testing the equation with  $(\bm{S}^{k,n,m,l}_{j},\bm{u}^{k,n,m,l}_{j})$, we see that:
\small
\begin{equation}\label{discrete_testing}
\int_\Omega \textcolor{black}{\bm{\mathcal{D}}}^k(\cdot,\bm{S}^{k,n,m,l}_{j}):\bm{S}^{k,n,m,l}_{j} + \frac{1}{l}\|\bm{u}_j^{k,n,m,l}\|^{2r'}_{L^{2r'}(\Omega)} \leq \langle\bm{f},\bm{u}^{k,n,m,l}_j\rangle + \frac{1}{\tau_m}\int_\Omega\bm{u}^{k,n,m,l}_{j-1}\cdot \bm{u}^{k,n,m,l}_{j}.
\end{equation}
\normalsize
On the other hand, since all norms are equivalent in a finite-dimensional normed linear space, there is a constant $C_n>0$ such that:
\begin{equation}\label{normeq_discrete}
\|\bm{v}\|_{W^{1,r}(\Omega)} \leq C_n \|\bm{v}\|_{L^{2r'}(\Omega)}\qquad \forall\, \bm{v}\in V^n_\text{div}.
\end{equation}
The constant $C_n$ may blow up as $n\rightarrow\infty$, but since $n$ is fixed for now this does not pose a problem. Now, recalling \eqref{As_Sel_Aprr2} and combining \eqref{discrete_testing} and \eqref{normeq_discrete} with a standard corollary of Brouwer's Fixed Point Theorem (cf.\ \cite{Girault1986}) we obtain the existence of solutions of Formulation $\check{\text{A}}_{\text{k},\text{n},\text{m},\text{l}}$. In the first time step (i.e.\ $j=1$), it is essential to use the fact that the projection $P^n_\text{div}$ is stable:
\begin{equation}
\|P^n_\text{div}\bm{u}_0\|_{L^2(\Omega)} \leq \|\bm{u}_0\|_{L^2(\Omega)}.
\end{equation}
The estimate \eqref{normeq_discrete} suffices to guarantee the existence of discrete solutions, but in order to pass to the limit $n\rightarrow\infty$, an estimate that does not degenerate as $n\rightarrow\infty$ is required. This uniform estimate is a consequence of the discrete inf-sup condition \eqref{infsupStress_A}:
\begin{equation}
\gamma_r\|\bm{u}^{k,n,m,l}_{j}\|_{W^{1,r}(\Omega)} \leq \|\textcolor{black}{\bm{\mathcal{D}}}^k(\cdot,\bm{S}^{k,n,m,l}_{j+1})\|_{L^r(\Omega)}.
\end{equation}
Therefore, the following a priori estimate holds:
\small
\begin{align}
&\sup_{j\in\{1,\ldots,T/\tau_m\}}\|\bm{u}^{k,n,m,l}_j\|^2_{L^2(\Omega)} + \sum_{j=1}^{T/\tau_m}\|\bm{u}^{k,n,m,l}_j-\bm{u}^{k,n,m,l}_{j-1}\|^2_{L^2(\Omega)}\notag\\ &+ \tau_m\sum_{j=1}^{T/\tau_m}\|\bm{S}^{k,n,m,l}_j\|_{L^{r'}(\Omega)} + \tau_m\sum_{j=1}^{T/\tau_m}\|\bm{u}^{k,n,m,l}_j\|^r_{W^{1,r}(\Omega)} \label{apriori_unsteady1}\\&+ \sum_{j=1}^{T/\tau_m}\|\textcolor{black}{\bm{\mathcal{D}}}^k(\cdot,\cdot,\bm{S}^{k,n,m,l}_j)\|_{L^r(Q_{j-1}^j)} + \frac{\tau_m}{l}\sum_{j=1}^{T/\tau_m} \|\bm{u}^{k,n,m,l}_j\|^{2r'}_{L^{2r'}(\Omega)} \leq c,\notag
\end{align}
\normalsize
where $c$ is a positive constant that depends on the data; in particular, $c$ is independent of $k,n,m$ and $l$. Let $\overline{\bm{u}}^{k,n,m,l}\in L^\infty(0,T;V^n_\text{div})$ and $\tilde{\bm{u}}^{k,n,m,l}\in C([0,T];V^n_\text{div})$ be the piecewise constant and piecewise linear interpolants defined by the sequence $\{\bm{u}^{k,n,m,l}_j\}_{j=1}^{T/\tau_m}$ (see \eqref{pwise_constant_interp} and \eqref{pwise_linear_interp}) and let $\overline{\bm{S}}^{k,n,m,l}\in L^\infty(0,T;\Sigma^n_\text{sym})$ be the piecewise constant interpolant defined by the sequence $\{\bm{S}^{k,n,m,l}_j\}_{j=1}^{T/\tau_m}$. Furthermore, define also the piecewise constant interpolants:
\begin{equation*}
\overline{\bm{f}}(t,\cdot) := \bm{f}_j(\cdot),\qquad \overline{\textcolor{black}{\bm{\mathcal{D}}}}^{k}(t,\cdot,\cdot) := \textcolor{black}{\bm{\mathcal{D}}}_j^{k}(\cdot,\cdot),\quad t\in (t_{j-1},t_j], \quad j\in\{1,\ldots,T/\tau_m\} 
\end{equation*}
Then the discrete formulation can be rewritten as:
\small
\begin{align*}
\int_\Omega (\overline{\textcolor{black}{\bm{\mathcal{D}}}}^k(t,\cdot,\overline{\bm{S}}^{k,n,m,l}) - \bm{D}(\overline{\bm{u}}^{k,n,m,l})):\bm{\tau} &= 0 &\forall\, \bm{\tau}\in \Sigma^n_{\text{sym}},\\
\int_\Omega \partial_t\tilde{\bm{u}}^{k,n,m,l}\cdot \bm{v}  + \frac{1}{l}\int_\Omega |\overline{\bm{u}}^{k,n,m,l}|^{2r'-2}\overline{\bm{u}}^{k,n,m,l}\cdot\bm{v}& & \\ +\int_\Omega   (\overline{\bm{S}}^{k,n,m,l}:\bm{D}(\bm{v})  + \mathcal{B}(\overline{\bm{u}}^{k,n,m,l}&,\overline{\bm{u}}^{k,n,m,l},\bm{v}))  = \langle \overline{\bm{f}},\bm{v}\rangle &\forall\, \bm{v}\in V^n_\text{div},\\
\tilde{\bm{u}}^{k,n,m,l}(0,\cdot) = P^n_\text{div}\bm{u}_0(\cdot). & &
\end{align*}
\normalsize
The a priori estimate \eqref{apriori_unsteady1} can in turn be written as:
\small
\begin{align}
&\|\overline{\bm{u}}^{k,n,m,l}\|^2_{L^\infty(0,T;L^2(\Omega))} + \tau_m\|\partial_t \tilde{\bm{u}}^{k,n,m,l}\|^2_{L^2(Q)}+ \|\overline{\bm{S}}^{k,n,m,l}\|^{r'}_{L^{r'}(Q)} \label{apriori1_bis_unsteady}\\
&+\|\overline{\bm{u}}^{k,n,m,l}\|^r_{L^r(0,T;W^{1,r}(\Omega))} + \|\textcolor{black}{\bm{\mathcal{D}}}^k(\cdot,\cdot,\overline{\bm{S}}^{k,n,m,l})\|^r_{L^r(Q)} + \frac{1}{l}\|\overline{\bm{u}}^{k,n,m,l}\|^{2r'}_{L^{2r'}(Q)}\leq c.\notag
\end{align}
\normalsize
Using the equivalence of norms in finite-dimensional spaces we also obtain
\begin{equation*}
\|\partial_t\tilde{\bm{u}}^{k,n,m,l}\|_{L^\infty(0,T;L^2(\Omega))} \leq c(n)\|\partial_t\tilde{\bm{u}}^{k,n,m,l}\|_{L^2(Q)}, 
\end{equation*}
and together with the a priori estimate this implies that
\begin{equation}
\|\tilde{\bm{u}}^{k,n,m,l}\|_{W^{1,\infty}(0,T;L^2(\Omega))} \leq c(n,m).
\end{equation}
Therefore, up to subsequences, as $k\rightarrow\infty$ we have:
\small
\begin{align}
\overline{\bm{u}}^{k,n,m,l} &\rightarrow \overline{\bm{u}}^{n,m,l}\qquad &\text{strongly in }L^\infty(0,T;L^2(\Omega)^d),\notag\\
\tilde{\bm{u}}^{k,n,m,l} &\rightarrow \tilde{\bm{u}}^{n,m,l}\qquad &\text{strongly in }W^{1,\infty}(0,T;L^2(\Omega)^d),\notag \\
\overline{\bm{u}}^{k,n,m,l} &\rightarrow \overline{\bm{u}}^{n,m,l}\qquad &\text{strongly in }L^{2r'}(Q)^d,\notag\\
\overline{\bm{u}}^{k,n,m,l} &\rightarrow \overline{\bm{u}}^{n,m,l}\qquad &\text{strongly in }L^r(0,T;W^{1,r}_0(\Omega)^d),\notag\\
\overline{\bm{S}}^{k,n,m,l} &\rightarrow \overline{\bm{S}}^{n,m,l}\qquad &\text{strongly in }L^{r'}(Q)^{d\times d},\notag\\
\textcolor{black}{\bm{\mathcal{D}}}^k(\cdot,\cdot,\overline{\bm{S}}^{k,n,m,l}) &\rightharpoonup \bm{D}^{n,m,l}
\qquad &\text{weakly in }L^r(Q)^{d\times d},\notag\\
\overline{\textcolor{black}{\bm{\mathcal{D}}}}^k(\cdot,\cdot,\overline{\bm{S}}^{k,n,m,l}) &\rightharpoonup \overline{\bm{D}}^{n,m,l}
\qquad &\text{weakly in }L^r(Q)^{d\times d},\notag\\
\textcolor{black}{\bm{\mathcal{D}}}^k_j(\cdot,\bm{S}_j^{k,n,m,l}) &\rightharpoonup \bm{D}_j^{n,m,l}
\qquad &\text{weakly in }L^r(\Omega)^{d\times d},\,\text{for }j\in\{1,\ldots,T/\tau_m\}.\notag
\end{align}
\normalsize
Since the function $\bm{D}^k_j$ is simply an average in time, the uniqueness of the weak limit implies that
\begin{equation}\label{pwise_D_conv}
\bm{D}^{n,m,l}_j(\cdot) = \frac{1}{\tau_m}\int_{t_{j-1}}^{t_j}\bm{D}^{n,m,l}(t,\cdot)\,\text{d}t,\qquad j\in \{1,\ldots,T/\tau_m\},
\end{equation}
and that $\overline{\bm{D}}^{n,m,l}$ is the piecewise constant interpolant determined by the sequence $\{\bm{D}^{n,m,l}_j\}_{j=1}^{T/\tau_m}$. Moreover, since the convergence of the velocity and stress sequences is strong, it is straightforward to pass to the limit $k\rightarrow\infty$ and thus we obtain
\small
\begin{align*}
\int_\Omega (\overline{\bm{D}}^{n,m,l} - \bm{D}(\overline{\bm{u}}^{n,m,l})):\bm{\tau} &= 0 &\forall\, \bm{\tau}\in \Sigma^n_{\text{sym}},\\
\int_\Omega \partial_t\tilde{\bm{u}}^{n,m,l}\cdot \bm{v}  + \frac{1}{l}\int_\Omega |\overline{\bm{u}}^{n,m,l}|^{2r'-2}\,\overline{\bm{u}}^{n,m,l}&\cdot\bm{v} & \\ +\int_\Omega   (\overline{\bm{S}}^{n,m,l}:\bm{D}(\bm{v})  + \mathcal{B}(\overline{\bm{u}}^{n,m,l} &,\overline{\bm{u}}^{n,m,l},\bm{v}))  = \langle \overline{\bm{f}},\bm{v}\rangle & \forall\, \bm{v}\in V^n_\text{div}.
\end{align*}
\normalsize
It is also clear that the initial condition $\tilde{\bm{u}}^{n,m,l}(0,\cdot) = P^n_\text{div}\bm{u}_0(\cdot)$ holds, since the expression on the right-hand side is independent of $k$. The identification of the constitutive relation can be carried out using \eqref{As_Sel_Aprr3} in exactly the same manner as in \cite{Sueli2018}, which means that (the strong convergence is again essential):
\begin{equation}
\textcolor{black}{(\bm{D}^{n,m,l},\overline{\bm{S}}^{n,m,l})\in \mathcal{A}(\cdot), \text{ a.e. in } (0,T)\times\Omega.}
\end{equation}
The next step is to take the limit in both the time and space discretisations simultaneously. The weak lower semicontinuity of the norms and the estimate \eqref{apriori1_bis_unsteady} imply that:
\begin{align}
&\|\overline{\bm{u}}^{n,m,l}\|^2_{L^\infty(0,T;L^2(\Omega))} + \tau_m\|\partial_t \tilde{\bm{u}}^{n,m,l}\|^2_{L^2(Q)}+ \|\overline{\bm{S}}^{n,m,l}\|^{r'}_{L^{r'}(Q)} \label{apriori2_unsteady}\\
&+\|\overline{\bm{u}}^{n,m,l}\|^r_{L^r(0,T;W^{1,r}(\Omega))} + \|\bm{D}^{n,m,l}\|^r_{L^r(Q)} + \frac{1}{l}\|\overline{\bm{u}}^{n,m,l}\|^{2r'}_{L^{2r'}(Q)}\leq c,\notag
\end{align}
and
\begin{equation}
\|\tilde{\bm{u}}^{n,m,l}\|^2_{L^\infty(0,T;L^2(\Omega))} = \|\overline{\bm{u}}^{n,m,l}\|^2_{L^\infty(0,T;L^2(\Omega))} \leq c,
\end{equation}
where $c$ is a constant, independent of $n,m$ and $l$. Consequently, there exist (not relabelled) subsequences such that, as $n,m\rightarrow\infty$:
\small
\begin{align}
\overline{\bm{u}}^{n,m,l} &\overset{\ast}{\rightharpoonup} \bm{u}^{l}\qquad &\text{weakly* in }L^\infty(0,T;L^2(\Omega)^d),\notag\\
\tilde{\bm{u}}^{n,m,l} &\overset{\ast}{\rightharpoonup} \bm{u}^{l}\qquad &\text{weakly* in }L^\infty(0,T;L^2(\Omega)^d),\notag\\
\overline{\bm{u}}^{n,m,l} &\rightharpoonup \bm{u}^{l}\qquad &\text{weakly in }L^r(0,T;W^{1,r}_0(\Omega)^d),\notag\\
\overline{\bm{S}}^{n,m,l} &\rightharpoonup \bm{S}^{l}\qquad &\text{weakly in }L^{r'}(Q)^{d\times d},\notag\\
\bm{D}^{n,m,l} &\rightharpoonup \bm{D}^{l}
\qquad &\text{weakly in }L^r(Q)^{d\times d},\notag\\
\overline{\bm{D}}^{n,m,l} &\rightharpoonup \overline{\bm{D}}^{l}
\qquad &\text{weakly in }L^r(Q)^{d\times d},\notag\\
\frac{1}{l}\int_Q |\overline{\bm{u}}^{n,m,l}|^{2r'-2}\overline{\bm{u}}^{n,m,l} &\rightharpoonup \frac{1}{l}\int_Q |\bm{u}^{l}|^{2r'-2}\bm{u}^{n,m,l}
\qquad &\text{weakly in }L^{(2r')'}(Q)^{d}.\notag
\end{align}
\normalsize
At this point it is a standard step to use the Aubin--Lions lemma to obtain strong convergence of subsequences. However, following \cite{Sueli2018}, we will instead use Simon's compactness lemma; this choice is made to avoid the need for stability estimates of $P^n_\text{div}$ in Sobolev norms, which would require additional assumptions on the mesh. To apply this lemma, it will be more convenient to work with the modified interpolant:
\begin{equation*}
\renewcommand{\arraystretch}{1.5}
\hat{\bm{u}}^{n,m,l}(t,\cdot) := \left\{
\begin{array}{cc}
\bm{u}^{n,m,l}_1(\cdot), & \textrm{ if } t\in[0,t_1),\\
\tilde{\bm{u}}^{n,m,l}(t,\cdot), & \textrm{ if }t\in [t_1,T].\\
\end{array}
\right.
\end{equation*}
Let $\epsilon >0$ be such that $s+\epsilon < T$ and let $\bm{v}\in V^n_\text{div}$. Then, using the definition of $\hat{\bm{u}}^{n,m,l}$ we have
\small
\begin{align*}
&\int_\Omega (\hat{\bm{u}}^{n,m,l}(s+\epsilon,x) - \hat{\bm{u}}^{n,m,l}(s+\epsilon,x))\cdot \bm{v}(x)\,\text{d}x\\& = \int_{\max(s,\tau_m)}^{s+\epsilon}\int_\Omega \partial_t\hat{\bm{u}}^{n,m,l}(t,x)\cdot\bm{v}(x)\,\text{d}x\,\text{d}t\\
&=\int_{\max(s,\tau_m)}^{s+\epsilon}\int_\Omega \partial_t\tilde{\bm{u}}^{n,m,l}(t,x)\cdot\bm{v}(x)\,\text{d}x\,\text{d}t \\&= \int_{\max(s,\tau_m)}^{s+\epsilon}\left(-\frac{1}{l}\int_\Omega |\overline{\bm{u}}^{n,m,l}(t,x)|^{2r'-2}\overline{\bm{u}}^{n,m,l}(t,x)\cdot\bm{v}(x)\,\text{d}x\right. \\&\left. -\int_\Omega   (\overline{\bm{S}}^{n,m,l}(t,x):\bm{D}(\bm{v}(x))  + \mathcal{B}(\overline{\bm{u}}^{n,m,l}(t,x),\overline{\bm{u}}^{n,m,l}(t,x),\bm{v}(x)))\, \text{d}x  + \langle \overline{\bm{f}}(t),\bm{v}\rangle \right)\,\text{d}t\\
&\leq c(l) \left(\left(\int_{\max(s,\tau_m)}^{s+\epsilon} \|\bm{v}\|^r_{W^{1,r}(\Omega)}\,\text{d}t  \right)^{1/r} +  \left(\int_{\max(s,\tau_m)}^{s+\epsilon} \|\bm{v}\|^{2r'}_{L^{2r'}(\Omega)}  \,\text{d}t\right)^{1/2r'} \right)\\
&\leq c(l)(\epsilon^{1/r} + \epsilon^{1/2r'})\left(\|\bm{v}\|_{W^{1,r}(\Omega)} + \|\bm{v}\|_{L^{2r'}(\Omega)}\right).
\end{align*}
\normalsize
Choosing $\bm{v} = \hat{\bm{u}}^{n,m,l}(s+\epsilon,\cdot)-\hat{\bm{u}}^{n,m,l}(s,\cdot)$ we conclude that
\begin{equation*}
\int_0^{T-\epsilon} \|\hat{\bm{u}}^{n,m,l}(s+\epsilon,\cdot)-\hat{\bm{u}}^{n,m,l}(s,\cdot)\|^2_{L^2(\Omega)}\,\text{d}s \rightarrow 0,\text{ as }\epsilon\rightarrow 0.
\end{equation*}
On the other hand, the a priori estimates imply that $\hat{\bm{u}}^{n,m,l}$ is bounded (uniformly in $n,m\in\mathbb{N}$) in $L^2(Q)^d$ and $L^1(0,T;W^{1,r}_0(\Omega)^d)$. Moreover, since $r>\frac{2d}{d+2}$, the embedding $W^{1,r}(\Omega)^d \hookrightarrow L^2(\Omega)^d$ is compact and thus Simon's compactness lemma guarantees the strong convergence:
\begin{equation}
\hat{\bm{u}}^{n,m,l} \rightarrow \bm{u}^{l}\qquad \text{strongly in }L^2(Q)^d.
\end{equation}
Since the interpolants converge to the same limit as $\tau_m\rightarrow 0$, using standard function space interpolation (and recalling \eqref{parabolic_embedding_r}) we also obtain that, as $n,m\rightarrow\infty$:
\begin{align}
\tilde{\bm{u}}^{n,m,l} &\rightarrow \bm{u}^{l}\qquad &\text{strongly in }L^p(0,T;L^{2}(\Omega)^d),\label{strong_unsteady_tilde}\\
\overline{\bm{u}}^{n,m,l} &\rightarrow \bm{u}^{l}\qquad &\text{strongly in }L^p(0,T;L^{2}(\Omega)^d)\cap L^{q}(Q),\label{strong_unsteady}
\end{align}
for $p\in [1,\infty)$ and $q\in [1,\max(2r',\frac{q(d+2)}{d}))$.

Now, using the property \eqref{Approx_FEMconv}, we can check that $\bm{u}^l$ is actually divergence-free:
\small
\begin{equation}\label{limit_divfree}
0 = \int_0^T\int_\Omega \phi\,\Pi^n_M q\,\text{div}\overline{\bm{u}}^{n,m,l} \rightarrow \int_0^T\int_\Omega \phi\, q\,\text{div}\bm{u}^l\quad\forall\, q\in L^{r'}(\Omega),\,\phi\in C^\infty_0(0,T).
\end{equation}
\normalsize
Furthermore, \eqref{Approx_FEMconv} also yields convergence of the initial condition, as $n,m\rightarrow\infty$:
\begin{equation}
\tilde{\bm{u}}^{n,m,l}(0,\cdot) = P^n_\text{div}\bm{u}_0 \rightarrow \bm{u}_{0}\qquad \text{strongly in }L^2(\Omega)^d.
\end{equation}
The functions $\bm{D}^{l}$ and $\overline{\bm{D}}^l$ can easily be identified using the property \eqref{conv_time_interpolant} and the definition of the piecewise constant interpolant \eqref{pwise_D_conv}. Indeed, for an arbitrary $\bm{\sigma}\in C^\infty_0(Q)$ we have, as $n,m\rightarrow\infty$:
\begin{equation}
\int_0^T\int_\Omega \overline{\bm{D}}^{n,m,l}:\bm{\sigma} = \int_0^T\int_\Omega \bm{D}^{n,m,l}:\overline{\bm{\sigma}} \rightarrow \int_0^T\int_\Omega \bm{D}^l:\bm{\sigma}.
\end{equation}
The uniqueness of the weak limit then implies that $\bm{D}^l=\overline{\bm{D}}^l$.

Combining all these properties and using an analogous computation to \eqref{limit_divfree} it is possible to prove that the limiting functions are a solution of the following problem:
\small
\begin{gather*}
\hspace{0.7cm}\int_0^T\int_\Omega (\bm{D}^{l} - \bm{D}(\bm{u}^{l})):\bm{\tau}\,\varphi = 0 \hspace{2cm}\forall\, \bm{\tau}\in C^\infty_{0,\text{sym}}(\Omega)^{d\times d},\,\varphi\in C^\infty_0(0,T),\\
\hspace{-2.6cm}-\int_0^T\int_\Omega \bm{u}^{l}\cdot \bm{v}\,\partial_t\varphi - \int_\Omega \bm{u}_0\cdot \bm{v}\varphi(0)  +  \int_0^T\int_\Omega   (\bm{S}^{l} - \bm{u}^l\otimes \bm{u}^l):\bm{D}(\bm{v})\,\varphi\\ \hspace{0.6cm} +\frac{1}{l}\int_0^T\int_\Omega |\bm{u}^{l}|^{2r'-2}\bm{u}^{l}\cdot\bm{v}\,\varphi  = \int_0^T\langle \bm{f},\bm{v}\rangle\,\varphi \hspace{0.8cm} \forall\, \bm{v}\in C^\infty_{0,\text{div}}(\Omega)^d,\,\varphi\in C^\infty_0(-T,T).
\end{gather*}
\normalsize
From the equation above and the \textcolor{black}{estimate} \eqref{parabolic_embedding_r} we then see that the distributional time derivative belongs to the spaces:
\begin{gather}
\partial_t \bm{u}^l \in L^{\min(r',(2r')')}(0,T;(W^{1,r}_{0,\text{div}}(\Omega)^d\cap L^{2r'}(\Omega)^d)^*),\label{time_der_space0}\\
\partial_t \bm{u}^l \in L^{\min(\check{r},(2r')')}(0,T;(W^{1,\check{r}'}_{0,\text{div}}(\Omega)^d)^*).\label{time_der_space}
\end{gather}
It is important to note that \eqref{time_der_space} holds uniformly in $l\in\mathbb{N}$, while \eqref{time_der_space0} does not. Now, observe that
\small
\begin{equation*}
W^{1,r}_{0,\text{div}}(\Omega)^d\cap L^{2r'}(\Omega)^d\hookrightarrow L^2_\text{div}(\Omega)^d  \hookrightarrow (L^2_\text{div}(\Omega)^d)^* \hookrightarrow (W^{1,r}_{0,\text{div}}(\Omega)^d\cap L^{2r'}(\Omega)^d)^*.
\end{equation*}
\normalsize
Combining this with \eqref{cont_time_embedding1}, \eqref{cont_time_embedding2}, and the fact that $\bm{u}^l\in L^\infty(0,T;L^2_\text{div}(\Omega)^d)$ guarantees that $\bm{u}^l\in C_w([0,T],L^2_\text{div}(\Omega)^d)$. Let $\bm{v}\in C^\infty_{0,\text{div}}(\Omega)^d$ and $\varphi\in C^\infty(-T,T)$ be such that $\varphi(0)=1$; then the following equality holds:
\begin{equation}\label{in_cond1}
\int_0^T\int_\Omega \partial_t(\bm{u}^l\varphi)\cdot \bm{v} = -\int_\Omega \bm{u}^{l}(0,\cdot)\cdot\bm{v}\,\varphi(0). 
\end{equation}
On the other hand, using the equation we also have that:
\small
\begin{equation}\label{in_cond2}
\int_0^T\int_\Omega \partial_t(\bm{u}^l\varphi)\cdot \bm{v} = \int_0^T\int_\Omega \partial_t\bm{u}^l\cdot \bm{v}\,\varphi +  \int_0^T\int_\Omega \bm{u}^l\cdot \bm{v}\,\partial_t\varphi = -\int_\Omega \bm{u}_{0}\cdot\bm{v}\,\varphi(0).
\end{equation}
\normalsize
Comparing \eqref{in_cond1} and \eqref{in_cond2} we conclude that $\bm{u}^l(0,\cdot) = \bm{u}_0(\cdot)$. This proves that the initial condition is attained in the weak sense expected a priori from the embeddings; however, in this case the stronger condition
\begin{equation}\label{stong_in_cond1}
\esslim_{t\rightarrow 0^+} \, \|\bm{u}^l(t,\cdot) - \bm{u}_0(\cdot) \|_{L^2(\Omega)} = 0
\end{equation}
holds. To see this, note that \eqref{strong_unsteady_tilde} guarantees that, up to a subsequence, $\tilde{\bm{u}}^{n,m,l}(t,\cdot)\rightarrow \tilde{\bm{u}}^{l}(t,\cdot)$ in $L^2(\Omega)^d$ for almost every $t\in [0,T]$, and therefore
\small
\begin{align*}
&\|\bm{u}^l(t,\cdot) - \bm{u}_0(\cdot) \|^2_{L^2(\Omega)} = \displaystyle\limsup_{n,m\rightarrow\infty} \|\tilde{\bm{u}}^{n,m,l}(t,\cdot) - \tilde{\bm{u}}^{n,m,l}(0,\cdot) \|^2_{L^2(\Omega)} \\
&= \limsup_{n,m\rightarrow\infty}\left(\|\tilde{\bm{u}}^{n,m,l}(t,\cdot)\|^2_{L^2(\Omega)} - \|\tilde{\bm{u}}^{n,m,l}(0,\cdot) \|^2_{L^2(\Omega)}\right.\\&\qquad \left. + 2\int_\Omega (\tilde{\bm{u}}^{n,m,l}(0,\cdot) - \tilde{\bm{u}}^{n,m,l}(t,\cdot))\cdot \tilde{\bm{u}}^{n,m,l}(0,\cdot)\right)\\
&\leq \limsup_{n,m\rightarrow\infty}\left(\int_0^t \langle \overline{\bm{f}},\overline{\bm{u}}^{n,m,l}\rangle + 2\int_\Omega (\tilde{\bm{u}}^{n,m,l}(0,\cdot) - \tilde{\bm{u}}^{n,m,l}(t,\cdot))\cdot \tilde{\bm{u}}^{n,m,l}(0,\cdot)\right)\\
&\leq  \int_0^t \langle \bm{f},\bm{u}^{l}\rangle+2\int_\Omega (\bm{u}^{l}(0,\cdot) - \bm{u}^{l}(t,\cdot))\cdot \bm{u}^{l}(0,\cdot),
\end{align*}
\normalsize
for almost every $t\in [0,T]$. Observe also that the monotonicity of the constitutive relation was used to obtain the next to last inequality. Taking the limit $t\rightarrow 0^+$ then yields \eqref{stong_in_cond1}.

The identification of the constitutive relation, \textcolor{black}{i.e.\ proving that $(\bm{D}^l,\bm{S}^l)\in\mathcal{A}(\cdot)$ almost everywhere}, can be carried out with the help of Lemma \ref{LocalizedMinty}. In order to apply the lemma, the only thing that remains to be proved, since we already know that $(\bm{D}^{n,m,l},\overline{\bm{S}}^{n,m,l})\in \mathcal{A}(\cdot)$ almost everywhere, is that:
\begin{equation}
\limsup_{n,m\rightarrow\infty}\int_0^t\int_\Omega \overline{\bm{S}}^{n,m,l}:\bm{D}^{n,m,l} \leq \int_0^t\int_\Omega \bm{S}^l:\bm{D}^l,
\end{equation}
for almost every $t\in[0,T]$; then taking $t\rightarrow T$ we obtain the result in the whole domain $Q$. The proof of this fact is essentially the  same as in \cite{Sueli2018} and we will not reproduce it here. 
Moreover, the following energy identity holds: 
\small
\begin{equation}\label{energy_equality}
\frac{1}{2}\|\bm{u}^l(t,\cdot)\|^2_{L^2(\Omega)} + \int_0^t\int_\Omega\bm{S}^l:\bm{D}(\bm{u}^l) + \frac{1}{l}\int_0^t\|\bm{u}^l\|^{2r'}_{L^{2r'}(\Omega)} = \int_0^t \langle  \bm{f},\bm{u}^l\rangle + \|\bm{u}_0\|^2_{L^2(\Omega)},
\end{equation}
\normalsize
In time-dependent problems obtaining an energy identity of this kind is not always possible; in this case the energy equality \eqref{energy_equality} can be proved, since the velocity is an admissible test function in space thanks to the fact that its $L^{2r'}$ norm is under control (some mollification is needed to overcome the low integrability in time, see \cite{Tscherpel2018,Lions1969}).

Now, \eqref{apriori2_unsteady} and the weak and weak* lower semicontinuity of the norms imply that
\small
\begin{equation}\label{apriori3_unsteady}
\|\bm{u}^{l}\|^2_{L^\infty(0,T;L^2(\Omega))} + \|\bm{S}^{l}\|^{r'}_{L^{r'}(Q)} +\|\bm{u}^{l}\|^r_{L^r(0,T;W^{1,r}(\Omega))} + \|\bm{D}^{l}\|^r_{L^r(Q)} + \frac{1}{l}\|\bm{u}^{l}\|^{2r'}_{L^{2r'}(Q)}\leq c,
\end{equation}
\normalsize
where $c$ is a constant independent of $l$. From this we see that, up to subsequences, as $l\rightarrow\infty$:
\begin{align}
\bm{u}^{l} &\overset{\ast}{\rightharpoonup} \bm{u}\qquad &\text{weakly* in }L^\infty(0,T;L^2(\Omega)^d),\notag\\
\bm{u}^{l} &\rightharpoonup \bm{u}\qquad &\text{weakly in }L^r(0,T;W^{1,r}_0(\Omega)^d),\notag\\
\bm{S}^{l} &\rightharpoonup \bm{S}\qquad &\text{weakly in }L^{r'}(Q)^{d\times d},\label{converg_3_unsteady}\\
\bm{D}^l &\rightharpoonup \bm{D}
\qquad &\text{weakly in }L^r(Q)^{d\times d},\notag\\
\frac{1}{l}\int_{\textcolor{black}{Q}} |\bm{u}^{l}|^{2r'-2}\bm{u}^{l} &\rightarrow 0
\qquad &\text{strongly in }L^{1}(Q)^{d}.\notag
\end{align}
Furthermore, since $\check{r}\leq r'$ and $r>\frac{2d}{d+2}$, the embedding $W^{1,\check{r}'}_{0,\text{div}}(\Omega)^d \hookrightarrow L^2_\text{div}(\Omega)^d$ is compact and hence by the Aubin--Lions lemma (taking into account \eqref{time_der_space}) we have the strong convergence:
\begin{equation}\label{converg_3_unsteady_strong}
\bm{u}^{l} \rightarrow \bm{u}\qquad \text{strongly in }L^r(0,T;L^{2}_{\text{div}}(\Omega)^d).
\end{equation}
With the convergence properties \eqref{converg_3_unsteady} and \eqref{converg_3_unsteady_strong} it is then possible to pass to the limit and prove that the limiting functions satisfy:
\small
\begin{align*}
\int_\Omega (\bm{D} - \bm{D}(\bm{u})):\bm{\tau} &= 0 &\forall\, \bm{\tau}\in C^\infty_{0,\text{sym}}(\Omega)^{d\times d},\text{ a.e. }t\in (0,T),\\
\langle\partial_t \bm{u},\bm{v} \rangle + \int_\Omega (\bm{S} - \bm{u}\otimes\bm{u}):\bm{D}(\bm{v}) &= \langle \bm{f},\bm{v}\rangle &\forall\, \bm{v}\in C^\infty_{0,\text{div}}(\Omega)^d, \text{ a.e. }t\in (0,T).
\end{align*}
\normalsize
The same argument used to obtain \eqref{stong_in_cond1} can be used here to prove that the initial condition is attained in the strong sense:
\begin{equation}\label{stong_in_cond2}
\esslim_{t\rightarrow 0^+} \, \|\bm{u}(t,\cdot) - \bm{u}_0(\cdot) \|_{L^2(\Omega)} = 0.
\end{equation}
Moreover, since the penalty term vanishes in the limit $l\rightarrow\infty$, we can improve the integrability in time:
\begin{equation}
\partial_t \bm{u}^l \in L^{\check{r}}(0,T;(W^{1,\check{r}'}_{0,\text{div}}(\Omega)^d)^*).
\end{equation}
To show that $(\bm{D},\bm{S})\in\mathcal{A}(\cdot)$, \cref{LocalizedMinty} will once again be employed. The main difficulty at this stage, just like in the previous works \cite{Diening:2013,Sueli2018}, is that the velocity is no longer an admissible test function (and therefore we do not have an energy equality similar to \eqref{energy_equality}). The idea is now to work with Lipschitz truncations of the error $\bm{e}^l := \bm{u}^l - \bm{u}$; it should be noted however that in the present case we need to verify a number of additional hypotheses before \cref{LipschitzTr_unsteady} can be applied.

Note that equation \eqref{eq_error_LipTr} in \cref{LipschitzTr_unsteady} is written in divergence form. We then need to make a preliminary step and write the penalty term in this form (see \cite{Sueli2018}). Let $B_0\subset\subset \Omega$ be an arbitrary ball compactly contained in $\Omega$ and let $q\in[1,(2r')')$. Then from the standard theory of elliptic operators we know that for almost every $t\in[0,T]$ there is a unique $\bm{g}_3^l(t,\cdot)\in W^{2,q}(B_0)^d\cap W^{1,q}_0(B_0)$ such that:
\begin{gather}
\int_{B_0} \nabla \bm{g}^l_3(t,\cdot):\nabla\bm{v} = \frac{1}{l} \int_{B_0} |\bm{u}^l(t,\cdot)|^{2r'-2}\bm{u}^l(t,\cdot)\cdot\bm{v}\qquad \forall\,\bm{v}\in C^\infty_{0,\text{div}}(\Omega)^d,\notag\\
\|\bm{g}^l_3(t,\cdot)\|_{W^{2,q}(B_0)} \leq c \left\|\frac{1}{l}|\bm{u}^l(t,\cdot)|^{2r'-2}\bm{u}^l(t,\cdot) \right\|_{L^q(B_0)}.\notag
\end{gather}
This means in particular (by \eqref{converg_3_unsteady} and standard function space interpolation) that for a fixed time interval $I_0\subset\subset (0,T)$ we have:
\begin{equation}\label{converg_pen_term}
\bm{g}^{l}_3 \rightarrow \bm{0}\qquad \text{strongly in }L^q(I_0;W^{1,q}(B_0)^d),\quad \forall\, q\in [1,(2r')').
\end{equation}
Defining $Q_0:= I_0\times B_0$ and
\begin{gather}
\textcolor{black}{\bm{G}^l_1 := \bm{S}^l - \bm{S}},\notag\\
\bm{G}^l_2 := \bm{u}^l\otimes\bm{u}^l - \bm{u}\otimes\bm{u} -\nabla\bm{g}^l_3,\notag
\end{gather}
we readily see that the error $\bm{e}^l$ satisfies the equation
\begin{equation}
\int_{Q_0} \partial_t\bm{e}^l\cdot \bm{w} = \int_{Q_0} (\bm{G}_1^l + \bm{G}^l_2):\nabla \bm{w}\qquad \forall\, \bm{w}\in C^\infty_{0,\text{div}}(Q_0)^d. 
\end{equation}
Additionally, as a consequence of \eqref{converg_3_unsteady}, \eqref{converg_pen_term} and \eqref{converg_3_unsteady_strong} we  also have that for any $q\in [1,\min (\check{r},(2r')')$, the sequence $\bm{u}^l$ is bounded in $L^\infty(I_0;W^{1,q}(Q_0)^d)$ and that: 
\begin{align*}
\bm{G}_1^{l} &\rightharpoonup \bm{0}\qquad &\text{weakly in }L^{r'}(Q_0)^{d\times d},\notag\\
\bm{G}_2^{l} &\rightarrow \bm{0}\qquad &\text{strongly in }L^{q}(Q_0)^{d\times d},\notag\\
\bm{u}^{l} &\rightarrow \bm{u}\qquad &\text{strongly in }L^q(Q_0)^d.\notag
\end{align*}
Consequently, the assumptions of Lemma \ref{LipschitzTr_unsteady} are satisfied. It now suffices to prove for an arbitrary $\theta\in (0,1)$ that
\begin{equation}\label{unsteady_integral_limit}
\limsup_{l\rightarrow\infty} \int_{\frac{1}{8}Q_0}[(\bm{D}(\bm{u}^{l})-\textcolor{black}{\bm{\mathcal{D}}}(\cdot,\bm{S})):(\bm{S}^{l}-\bm{S})]^\theta \leq 0,
\end{equation}
\textcolor{black}{Once this has been shown, Chacon's biting lemma and Vitali's convergence theorem will imply, 
together with Lemma \ref{LocalizedMinty}, that $(\bm{D},\bm{S})\in\mathcal{A}(\cdot)$ almost everywhere in $\frac{1}{8}Q_0$ (see the details e.g.\ in \cite{Bulicek:2012}).} From here then the result follows by observing that $Q$ can be covered by a union of such cylinders (e.g.\ by using a Whitney covering).

In order to prove \eqref{unsteady_integral_limit}, first let $\mathcal{B}_{\lambda_{l,j}}\subset \Omega$ be the family of open sets and let $\{\bm{e}^{l,j}\}_{l,j\in\mathbb{N}}$ be the sequence of Lipschitz truncations described in Lemma \ref{LipschitzTr_unsteady}.
If we define
\begin{equation}
H^l(\cdot) := (\bm{D}(\bm{u}^{l})-\textcolor{black}{\bm{\mathcal{D}}}(\cdot,\bm{S})):(\bm{S}^{l}-\bm{S})\in L^1(Q),
\end{equation}
then we have by Hölder's inequality that
\begin{equation*}
\int_{\frac{1}{8}Q_0} |H^l|^\theta \leq |Q|^{1-\theta}\left(\int_{\frac{1}{8}Q_0\setminus\mathcal{B}_{\lambda_{l,j}}} H^l\right)^\theta + |\mathcal{B}_{\lambda_{l,j}}|^{1-\theta}\left(\int_{\frac{1}{8}Q_0 }H^l\right)^\theta.
\end{equation*}
The second term on the right-hand side can be dealt with easily, since $H^l$ is bounded uniformly in $L^1(Q)$ thanks to the a priori estimate \eqref{apriori3_unsteady}, and the properties described in \cref{LipschitzTr_unsteady} imply that 
\begin{equation}\label{Unsteady_II}
\limsup_{l\rightarrow\infty}|\mathcal{B}_{\lambda_{l,j}}|^{1-\theta} \leq \limsup_{l\rightarrow\infty}|\lambda^r_{l,j} \mathcal{B}_{\lambda_{l,j}}|^{1-\theta}\leq c  2^{-j(1-\theta)},\quad \text{for }j\geq j_0,
\end{equation}
where $c$ is a positive constant. For the first term, observe that
\begin{align*}
& \int_{\frac{1}{8}  Q_0\setminus\mathcal{B}_{\lambda_{l,j}}} H^l = \textcolor{black}{ \int_{\frac{1}{8}Q_0} H^l \,\zeta \,\mathds{1}_{\mathcal{B}^c_{\lambda_{l,j}}} }\\
&\hphantom{1}\textcolor{black}{= \int_{\frac{1}{8}Q_0} \bm{D}(\bm{e}^{l}):(\bm{S}^{l}-\bm{S}) \,\zeta \,\mathds{1}_{\mathcal{B}^c_{\lambda_{l,j}}} + \int_{\frac{1}{8}Q_0\setminus\mathcal{B}_{\lambda_{l,j}}} (\bm{D}(\bm{u})-\textcolor{black}{\bm{\mathcal{D}}}(\cdot,\bm{S})):(\bm{S}^{l}-\bm{S})}\\
&\hphantom{1}\textcolor{black}{\leq \left| \int_{\frac{1}{8}Q_0} \bm{D}(\bm{e}^{l,j}):\bm{G}^l_1 \,\zeta \,\mathds{1}_{\mathcal{B}^c_{\lambda_{l,j}}}\right| + \left|\int_{\frac{1}{8}Q_0} (\bm{D}(\bm{u})-\textcolor{black}{\bm{\mathcal{D}}}(\cdot,\bm{S})):(\bm{S}^{l}-\bm{S}) \right| }\\
&\hphantom{1} \textcolor{black}{+ \left| \int_{\mathcal{B}_{\lambda_{l,j}}} (\bm{D}(\bm{u})-\textcolor{black}{\bm{\mathcal{D}}}(\cdot,\bm{S})):(\bm{S}^{l}-\bm{S}) \right|},
\end{align*} 

\noindent
\textcolor{black}{where $\zeta \in C^\infty_{0,\text{div}}(\frac{1}{6}Q_0)$ is the function introduced in \cref{LipschitzTr_unsteady}. Taking $\limsup_{l\rightarrow\infty}$ the assertion follows by taking $j\rightarrow\infty$. In particular, we used for the first term \cref{LipschitzTr_unsteady} part 6, with $\bm{H}=\bm{0}$, for the second term the weak convergence of $\bm{S}^l$ and for the third term the fact that $\{\bm{S}^l\}_{l\in\mathbb{N}}$ is bounded, together with \eqref{Unsteady_II}.} To conclude the proof, note that the fact that $\bm{u}$ is divergence-free and Assumption (A6) imply that $\text{tr}(\bm{S}) = 0$, and so $\bm{S}\in L^{r'}_\text{sym}(\Omega)^{d\times d}\cap  L^{r'}_\text{tr}(\Omega)^{d\times d}$.
\end{proof}
\begin{remark}\label{remark_steady}
	Formulation $\check{\text{A}}_{\text{k},\text{n},\text{m},\text{l}}$ is a four-step approximation in which the indices $k,n,m,l$ refer to the approximation of the graph by smooth functions, the finite element discretisation, the discretisation in time, and the penalty term, respectively. The same approach can be used to define a 3-field formulation for the steady problem and the unsteady problem without convection and the proof remains valid with some simplifications; for instance, for the steady system without convective term, only the indices $k$ and $n$ are needed. Furthermore, in those cases the convergence of the sequence of discrete pressures can be guaranteed in the corresponding Lebesgue spaces. 
\end{remark}
\begin{remark}
	The argument used to prove the existence of the discrete solutions is more involved here than in the original works \cite{Diening:2013,Bulicek:2009}, because the coercivity with respect to $\|\bm{u}^{k,n,m,l}_j\|_{W^{1,r}(\Omega)}$ cannot be deduced from Formulation $\check{\text{A}}_{\text{k},\text{n},\text{m},\text{l}}$ by simply testing with the solution. An alternative approach could be to include in the equation an additional diffusion term of the form: 
	\begin{equation*}
		\frac{1}{k}\int_\Omega |\bm{D}(\bm{u}_j^{k,n,m,l})|^{r-2}\bm{D}(\bm{u}_j^{k,n,m,l}):\bm{D}(\bm{v}),
	\end{equation*}
	which would be completely acceptable if we only cared about the existence of weak solutions, but is undesirable from the point of view of the computation of the finite element approximations, since it introduces an additional nonlinearity in the discrete problem.
\end{remark}

\begin{remark}\label{remark_disc_lipsc}
	In the proof of \cref{Convergence_unsteady_FormA} the limits $k\rightarrow\infty$, $(n,m)\rightarrow\infty$ and $l\rightarrow\infty$ were taken successively. In contrast to the steady case considered in \cite{Diening:2013}, here it is not known whether we can take the limits at once. The result is likely to hold as well, but the proof would require a discrete version of the parabolic Lipschitz truncation, which is not available at the moment.
\end{remark}
\begin{remark}
	In case the \textcolor{black}{symmetric velocity gradient} is a quantity of interest, the approach presented here can be easily extended to a four-field formulation with unknowns $(\bm{D},\bm{S},\bm{u},p)$. The only additional assumption needed in that case would be an inf-sup condition of the form:
	\begin{equation}
		\inf_{\bm{\sigma}\in \Sigma^n_{\text{div}}(\bm{0})}\sup_{\bm{\tau}\in \Sigma^n_\text{sym}}  
		\frac{\int_\Omega \bm{\sigma}:\bm{\tau}}{\|\bm{\sigma}\|_{L^{s'}(\Omega)}\|\bm{\tau}\|_{L^{s}(\Omega)}} \geq \delta_s,
	\end{equation}
	where $\delta_s>0$ is independent of $n$.
\end{remark}
\section{Numerical experiments}
\textcolor{black}{According to the analysis carried out in the previous section, the addition of the penalty term is necessary when $r\in (\frac{2d}{d+2},\frac{3d+2}{d+2}]$. However, in the examples we observed that the method converges regardless of whether the penalty term is present or not. This could be an indication that the requirement to include this penalty term is only a technical obstruction and that there might be a different approach to showing convergence of the numerical method that could avoid its inclusion in the numerical method. On the other hand, it could also be the case that exact solutions with more severe singularities than the ones considered in our numerical experiments are needed to demonstrate pathological behaviour. In any case, it appears that in most applications the penalty term can be safely omitted and for this reason it is not discussed in the numerical examples below.}
\subsection{Carreau fluid and orders of convergence}\label{exp_ratesC}
The framework presented in this work is so broad that in general it is not possible to guarantee uniqueness of solutions; in particular it is not clear how error estimates could be obtained. However, as this computational example will show, the discrete formulations presented here appear to recover the expected orders of convergence in the cases where these orders are known.

In the first part of this numerical experiment we solved the steady problem without convection with the Carreau constitutive law (as stated in \cref{remark_steady}, the same 3-field approximation can be applied in this setting): 
\begin{equation}\label{carreau_exp}
	\bm{S}(\bm{D}) := 2\nu \left(\varepsilon^2 + |\bm{D}^2| \right)^{\frac{r-2}{2}}\bm{D},
\end{equation}
where $r\geq 1$ and $\varepsilon,\nu>0$. This is one of the most common non-Newtonian models that present a power-law structure (note that for $r=2$ we recover the Newtonian model), and has the advantage that it is not singular at the origin (i.e.\ when $\bm{D} = \bm{0}$), unlike the usual power-law constitutive relation. Observe that the constitutive relation is smooth, and therefore only the limit $n\rightarrow\infty$ is needed in the results from the previous section. The problem was solved on the unit square $\Omega = (0,1)^2$ with a Dirichlet boundary condition for the velocity defined so as to match the value of the exact solution, which was chosen as:
\begin{equation}\label{exact_sol_carreau}
	\bm{u}(\bm{x}) = |\bm{x}|^{a-1}(x_2,-x_1)^\text{T},\qquad p(\bm{x}) = |\bm{x}|^b,
\end{equation}
where $a,b$ are parameters used to control the smoothness of the solutions. Define the auxiliary function $\bm{F}:= \mathbb{R}^{d\times d}\rightarrow\mathbb{R}^{d\times d}_{\text{sym}}$ as:
\begin{equation}
	\bm{F}(\bm{B}):= (\varepsilon + |\bm{B}^{\text{sym}}|)^{\frac{r-2}{2}}\bm{B}^{\text{sym}},
\end{equation}
where $\bm{B}^{\text{sym}} := \frac{1}{2}(\bm{B} + \bm{B}^T)$. In \cite{Belenki:2012,Hirn2013} it was proved for systems of the form \eqref{carreau_exp} that if $\bm{F}(\bm{D}(\bm{u}))\in W^{1,2}(\Omega)^{d\times d}$ and $p\in W^{1,r'}(\Omega)$ then the following error estimates hold:
\begin{align*}
	\|\bm{F}(\bm{D}(\bm{u})) - \bm{F}(\bm{D}(\bm{u}^n))\|_{L^2(\Omega)} &\leq c h_n^{\min\{1,\frac{r'}{2}\}},\\
	\|p - p^n\|_{L^{r'}(\Omega)} &\leq c h_n^{\min\{\frac{2}{r'},\frac{r'}{2}\}}.
\end{align*}
In our case, the conditions $\bm{F}(\bm{D}(\bm{u}))\in W^{1,2}(\Omega)^{d\times d}$ and $p\in W^{1,r'}(\Omega)$ amount to requiring that $a>1$ and $b>\frac{2}{r} -1$. These parameters were then chosen to be $a=1.01$ and $b = \frac{2}{r} - 0.99$ in order to be close to the regularity threshold. We discretised this problem with the Scott--Vogelius element for the velocity and pressure and discontinuous piecewise polynomials for the stress variables:
\begin{align*}
	\Sigma^n &= \{\bm{\sigma}\in L^\infty(\Omega)^{d\times d}\, :\, \bm{\sigma}|_K\in\mathbb{P}_k(K)^{d\times d},\text{ for all }K\in \mathcal{T}_n\},\\
	V^n &= \{\bm{w}\in W^{1,r}(\Omega)^d\, :\,\bm{w}|_{\partial\Omega} = \bm{u},\, \bm{w}|_K\in\mathbb{P}_{k+1}(K)^d\text{ for all }K\in \mathcal{T}_n\},\\
	M^n &= \{q\in L^\infty(\Omega)\, :\, q|_k\in \mathbb{P}_{k}(K)\text{ for all }K\in \mathcal{T}_n\}.
\end{align*}
The problem was solved using \texttt{firedrake} \cite{Rathgeber2016} with $\nu = 0.5$, $\varepsilon = 10^{-5}$ and $k=1$ on a barycentrically refined mesh \textcolor{black}{(obtained using \texttt{gmsh} \cite{Gmsh})} to guarantee inf-sup stability. The discretised nonlinear problems were linearised using Newton's method \textcolor{black}{with the $L^2$ line search algorithm of PETSc \cite{PETSc,PETScLi}}; \textcolor{black}{the Newton solver was deemed to have converged when the Euclidean norm of the residual fell below $1\times 10^{-8}$.} The linear systems were solved \textcolor{black}{with a sparse direct solver} from the \texttt{umfpack} library \cite{Umfpack}. In the implementation, the uniqueness of the pressure was recovered not by using a zero mean condition but rather by orthogonalising against the nullspace of constants. The experimental orders of convergence in the different norms are shown in \cref{tbl:table1,tbl:table2} (note that the tables do not contain the values of the numerical error, but rather the order of convergence corresponding to the norm indicated in each column).
\begin{table}[tbhp]
	\begin{center}
		{
			\centering
			\caption{Experimental order of convergence for the steady problem without convection with $r=1.5$.}%
			\label{tbl:table1}%
			\begin{tabular}{ccccc}
				\toprule
				$h_n$ & $\|\bm{F}(\bm{D}(\bm{u}))\|_{L^2(\Omega)}$ & $\|\bm{u}\|_{W^{1,r}(\Omega)}$ &  $\|p\|_{L^{r'}(\Omega)}$  & $\|\bm{S}\|_{L^{r'}(\Omega)}$\\
				\midrule
				0.5  & 0.9075 & 1.0180 & 0.3647 & 0.6692\\
				0.25 & 0.9803 & 1.2160 & 0.5396 & 0.6697\\
				0.125 & 1.0023 & 1.2975 & 0.6565 & 0.6713\\
				0.0625 & 1.0062 & 1.3205 & 0.6706 & 0.6716\\
				0.03125 & 1.0071 & 1.3319 & 0.6715 & 0.6716\\
				\midrule
				Expected & 1.0 & - & 0.667 & -\\
				\bottomrule
			\end{tabular}
		}
	\end{center}
\end{table}  
\begin{table}[!htbp]%
	\centering
	\caption{Experimental order of convergence for the steady problem without convection with $r=1.8$.}%
	\label{tbl:table2}%
	\begin{tabular}{ccccc}
		\toprule
		$h_n$ & $\|\bm{F}(\bm{D}(\bm{u}))\|_{L^2(\Omega)}$ & $\|\bm{u}\|_{W^{1,r}(\Omega)}$ &  $\|p\|_{L^{r'}(\Omega)}$  & $\|\bm{S}\|_{L^{r'}(\Omega)}$\\
		\midrule
		0.5  & 0.9132 & 0.9361 & 0.4955 & 0.8434\\
		0.25 & 0.9826 & 1.0652 & 0.7271 & 0.8822\\
		0.125 &  1.0040 & 1.1073 & 0.8671 & 0.8948\\
		0.0625 & 1.0078 & 1.1167 & 0.8916 & 0.8966\\
		0.03125 & 1.0087 & 1.1197 & 0.8959 & 0.8968\\
		\midrule
		Expected & 1.0 & - & 0.889 & -\\
		\bottomrule
	\end{tabular}
\end{table}

From \cref{tbl:table1,tbl:table2} it can be seen that the algorithm recovers the expected orders of convergence. In the case of the stress we obtain the same order as for the pressure, which seems natural from the point of view of the equation. In \cite{Hirn2013} it is claimed that for $r<2$ the order of convergence for the velocity should be equal to 1; in our numerical simulations the experimental order of convergence seems to approach $\frac{2}{r}$, which is slightly larger than 1. This difference may be due to the fact that in \cite{Hirn2013} the author works with piecewise linear elements for the velocity while here  quadratic elements were employed.

In the second part of the experiment we employed again the Carreau constitutive law \eqref{carreau_exp}, but now considering the full system \eqref{PDE_goal}. The right-hand side, initial condition and boundary condition were chosen so as to match the ones defined by the exact solution:
\begin{equation*}
	\bm{u}(t,\bm{x}) = t|\bm{x}|^{a-1}(x_2,-x_1)^\text{T},\quad p(t,\bm{x}) = t^2|\bm{x}|^b.
\end{equation*}
In \cite{Eckstein2018}, the following error estimate for the approximation of time-dependent systems of this form, but without convection, was obtained for $r\in [\frac{2d}{d+2},\infty)$:
\begin{equation*}
	\|\bm{u} - \overline{\bm{u}}^{n,m}\|_{L^\infty(0,T;L^2(\Omega))} + \|\bm{F}(\bm{D}(\bm{u}))- \bm{F}(\bm{D}(\overline{\bm{u}}^{n,m}))\|_{L^2(Q)} \leq c\left(\tau_m + h_n^{\min\{1,\frac{2}{r}\}}\right),
\end{equation*}
assuming that $\bm{u}_0\in W^{1,r}_{0,\text{div}}(\Omega)^d$ and that the following additional regularity properties of the solution and the data hold:
\begin{gather*}
	\|\nabla\bm{F}(\bm{D}(\bm{u}_0))\|_{L^2(\Omega)} + \|\nabla\bm{S}(\bm{D}(\bm{u}_0))\|_{L^2(\Omega)} \leq c,\\
	\|\bm{u}\|_{W^{1,2}(0,T;L^2(\Omega))} + \|\bm{u}\|_{L^{2}(0,T;W^{2,2}(\Omega))} + \|\bm{F}(\bm{D}(\bm{u}))\|_{L^2(0,T;W^{1,2}(\Omega))}\leq c.
\end{gather*}
The same order of convergence was obtained in \cite{Berselli2015} for $r\in (\frac{3}{2},2]$ in 3D for a semi-implicit discretisation of the unsteady system with convection assuming that $\bm{u}_0\in W^{2,2}_{0,\text{div}}(\Omega)^d$, $\text{div}\,\bm{S}(\bm{D}(\bm{u}_0))\in L^2(\Omega)^d$ and that the slightly different regularity assumptions hold:
\begin{equation*}
	\|\partial_t\bm{u}\|_{L^\infty(0,T;L^2(\Omega))} + \|\bm{F}(\bm{D}(\bm{u}))\|_{W^{1,2}(Q)} + \|\bm{F}(\bm{D}(\bm{u}))\|_{L^{2((5r-6)/(2-r))}(0,T;W^{1,2(\Omega)})}\leq c.
\end{equation*}

The problem was solved until the final time $T=0.1$ with the same parameters as above; observe that this choice of parameters guarantees that the required regularity properties are satisfied. Table \ref{tbl:table_unsteady} shows the experimental order of convergence for $r=1.7$.
\begin{table}[!htbp]%
	\centering
	\caption{Experimental order of convergence for the full problem with $r=1.7$.}%
	\label{tbl:table_unsteady}%
	\begin{tabular}{ccccc}
		\toprule
		$h_n$ & $\tau_m$ & $\|\bm{F}(\bm{D}(\bm{u}))\|_{L^2(Q)}$ & $\|\bm{u}\|_{L^\infty(0,T;L^2(\Omega))}$\\
		\midrule
		0.5  & 0.001 & 0.9226 & 1.8703\\
		0.25 & 0.0005 & 0.9865 & 1.9564\\
		0.125 & 0.00025 & 1.0057 & 1.9497\\
		0.0625 & 0.000125 & 1.0084 & 1.9440\\
		0.03125 & 0.0000625 & 1.0075 & 1.9451\\
		\midrule
		Expected & & 1.0 & 1.0 \\
		\bottomrule
	\end{tabular}
	\end{table}
	The order of convergence for the natural norm $\|\bm{F}(\bm{D}(\bm{u}))\|_{L^2(Q)}$ agrees with the one expected from the theoretical results, while for the velocity we obtain a higher order. This is again likely to be due to the fact that quadratic elements were employed for the velocity variable, while the analysis was performed for linear elements.  
	\iftoggle{arxiv}{
			\subsection{Role of the penalty term}\label{section_penalty}
			In this computational experiment we investigate the role of the penalty term in the algorithm, to explore whether its presence is essential to ensure convergence. Similarly to Section \ref{exp_ratesC} we consider the steady problem first. The same exact solution was employed, because it allows us to carefully select its regularity. In this case Taylor--Hood elements were employed for the velocity and pressure, and discontinuous piecewise polynomials for the stress:
			\small
			\begin{align*}
				\Sigma^n &= \{\bm{\sigma}\in L^\infty(\Omega)^{d\times d}\, :\, \bm{\sigma}|_K\in\mathbb{P}_1(K)^{d\times d},\text{ for all }K\in \mathcal{T}_n\},\\
				V^n &= \{\bm{w}\in W^{1,r}(\Omega)^d\, :\,\bm{w}_\tau|_{\partial\Omega_1}=0,\,\bm{w}|_{\partial\Omega_2} = \bm{0},\, \bm{w}|_K\in\mathbb{P}_{2}(K)^d\text{ for all }K\in \mathcal{T}_n\},\\
				M^n &= \{q\in L^\infty(\Omega)\cap C(\overline{\Omega})\, :\, q|_k\in \mathbb{P}_{1}(K)\text{ for all }K\in \mathcal{T}_n\}.
			\end{align*}
			\normalsize
			This question is only relevant when the discretely divergence-free elements are not pointwise divergence-free, because otherwise the condition $r>\frac{2d}{d+2}$ is sufficient to allow us to pass to the limit in the convective term. In the steady case, without the penalty term and with elements that are not exactly divergence-free, the convergence of the finite element approximations still holds (see \cref{remark_steady}), but assuming the stronger assumption $r>\frac{2d}{d+1}$ (cf.\ \cite{Diening:2013}). According to this, the addition of the penalty term is then necessary in the convergence analysis when the elements are not exactly divergence-free and $r\in(\frac{2d}{d+2},\frac{2d}{d+1}]$. Table \ref{tbl:table_penaltyterm1} shows the experimental orders of convergence for $r=1.3$ (just like in Section \ref{exp_ratesC}, the table shows not the numerical error, but the experimental order of convergence).

			In the experiment for the time-dependent problem we chose in this case the steady state solution \eqref{exact_sol_carreau} with the same parameters described above and used it to define the initial and boundary conditions. In this case, our convergence analysis dictates that the addition of the penalty term is necessary when $r\in (\tfrac{2d}{d+2},\tfrac{3d+2}{d+2}]$; however, the result is expected to hold for $r\in (\tfrac{2(d+1)}{d+2},\tfrac{3d+2}{d+2}]$ as well (see \cref{remark_disc_lipsc}). We therefore chose a value of $r$ in the interval $(\tfrac{2d}{d+2},\tfrac{2(d+1)}{d+2}]$. The experimental orders of convergence for this case are shown in Table \ref{tbl:table_unsteady_penalty}.

			\begin{table}[!htpb]
				\caption{Experimental order of convergence for the steady problem with $r=1.3$.}%
				\label{tbl:table_penaltyterm1}
				\subfloat[][With penalty term]{
					\resizebox{6.2cm}{!}{\begin{tabular}{ccc}
							\toprule
							$h_n$ & $\|\bm{F}(\bm{D}(\bm{u}))\|_{L^2(\Omega)}$ &  $\|p\|_{L^{r'}(\Omega)}$\\
							\midrule
							0.5  & 0.97295673154 &  1.91148217955\\
							0.25 & 1.00506435728 &  0.470815332994  \\
							0.125 & 1.0089872966  & 0.51434542432 \\
							0.0625 & 1.00879694502  & 0.472841717098 \\
							0.03125 & 1.00895395592  & 0.463776304819 \\
							\midrule
							Expected & 1.0 & 0.461\\
							\bottomrule
						\end{tabular}
				}}
				\subfloat[][Without penalty term.]{
					\resizebox{6.2cm}{!}{\begin{tabular}{ccc}
							\toprule
							$h_n$ & $\|\bm{F}(\bm{D}(\bm{u}))\|_{L^2(\Omega)}$ &  $\|p\|_{L^{r'}(\Omega)}$\\
							\midrule
							0.5  & 0.96952971684 &  2.0714861746\\
							0.25 & 0.99106185672 &  0.75134653896  \\
							0.125 & 0.99936073804  & 0.53246973909 \\
							0.0625 & 1.00160288555  & 0.481960358342 \\
							0.03125 & 1.004106766  & 0.469496736474 \\
							\midrule
							Expected & 1.0 & 0.461\\
							\bottomrule
						\end{tabular}
				}}
				\vspace{-0.3cm}
			\end{table}
			\begin{table}[!htbp]%
				\centering
				\caption{Experimental order of convergence for the $\|\bm{F}(\bm{D}(\bm{u}))\|_{L^2(Q)}$ norm for the full problem with $r=1.3$.}%
				\label{tbl:table_unsteady_penalty}%
				\begin{tabular}{ccccc}
					\toprule
					$h_n$ & $\tau_m$ & With Penalty Term & Without Penalty Term\\
					\midrule
					0.5  & 0.005 & 9.73599147231 & 5.502165863559\\
					0.25 & 0.0025 & 1.008703378392 & 0.98183996942\\
					0.125 & 0.00125 & 1.00651090357 & 1.00190875446\\
					0.0625 & 0.000625 & 1.0154632500 & 1.00811647604\\
					0.03125 & 0.0003125 & 1.028436230 & 1.01326314547\\
					\midrule
					Expected & & 1.0 & 1.0 \\
					\bottomrule
				\end{tabular}

			\end{table}
		What we see in these experiments is that the method converges regardless of whether there is a penalty term or not. This could be an indication that the requirement to include this penalty term is only a technical obstruction and that there might be a different approach to showing convergence of the numerical method that could avoid it. On the other hand, it could also be the case that the point singularity present in the exact solution of this  model problem does not suffice to demonstrate pathological behaviour. In any case, we believe that in most applications the penalty term can be safely omitted.}

		\subsection{Navier--Stokes/Euler activated fluid}
		In this section we will consider the classical lid--driven cavity problem with the non--standard constitutive relation:
		\begin{equation}\label{Euler_NS_ConstRel}
			\renewcommand{\arraystretch}{1.5}
			\left\{\begin{array}{cc}
					\left\{
						\begin{array}{cc}
							\bm{D} = \delta_s\frac{\bm{S}}{|\bm{S}|} + \frac{1}{2\nu} \bm{S}, & \textrm{ if } |\bm{D}|\geq \delta_s,\\
							\bm{S} = 0, & \textrm{ if }|\bm{D}|<\delta_s,\\
					\end{array} \right. & \textrm{if }(x-\tfrac{1}{2})^2 + (y-\tfrac{1}{2})^2 \leq (\tfrac{3}{8})^2,\\
					\bm{D} = \frac{1}{2\nu} \bm{S}, & \text{ otherwise },
				\end{array}
			\right.
				\end{equation}
				where $\nu>0$ is the viscosity and $\delta_s\geq 0$. This is an example of an activated fluid that in the middle of the domain transitions between a Newtonian fluid (i.e.\ Navier--Stokes) and an inviscid fluid (i.e.\ Euler) depending on the magnitude of the \textcolor{black}{symmetric velocity gradient} (for a more thorough discussion of activated fluids see \cite{Blechta2019}). It is analogous to the Bingham constitutive equation for a viscoplastic fluid, but with the roles of the stress and \textcolor{black}{symmetric velocity gradient} interchanged; the fact that we can swap the roles of the stress and the \textcolor{black}{symmetric velocity gradient} in constitutive relations without any problem is a significant advantage of the framework presented here.

The problem was solved on the unit square $\Omega = (0,1)^2$ with the rest state as the initial condition and with the following boundary conditions:
\begin{align*}
\partial\Omega_1 &= (0,1)\times \{1\}, \qquad &\partial\Omega_2 := \partial\Omega \setminus \partial\Omega_1,\\
\bm{u} &= \bm{0}\qquad &\text{ on }(0,T)\times\partial\Omega_2,\\
\bm{u} &= (x^2(1-x)^216y^2,0)^\text{T}\qquad &\text{ on }(0,T)\times\partial\Omega_1.
\end{align*}
				
\textcolor{black}{Although \eqref{Euler_NS_ConstRel} has a complicated form, there is a continuous (in $\bm{D}$) selection available:}
				\begin{equation}\label{eq:Euler_NS_ConstRel_Sel}
\renewcommand{\arraystretch}{1.5}
\textcolor{black}{\bm{S}=\bm{\mathcal{S}}(x,y,\bm{D}) := }\left\{
\begin{array}{cc}
	\textcolor{black}{2\nu \left(|\bm{D}|-\delta_s\mathds{1}_{B_{3/8}(1/2)}(x,y)\right)^+ \frac{\bm{D}}{|\bm{D}|},} & \textcolor{black}{\textrm{ if } |\bm{D}|\neq 0,}\\
	\textcolor{black}{\bm{0}, }& \textcolor{black}{\textrm{ if }|\bm{D}|=0.}\\
\end{array}
\right.
\end{equation}

\textcolor{black}{While the selection stated in \eqref{eq:Euler_NS_ConstRel_Sel} is already continuous in $\bm{D}$, Newton's method requires Fréchet-differentiability of $\bm{\mathcal{S}}$ with respect to $\bm{D}$} and the constitutive law is not smooth when $|(x-\tfrac{1}{2},y-\tfrac{1}{2})|<\tfrac{3}{8}$; \textcolor{black}{therefore some regularisation was required for the purpose of applying Newton's method (an alternative would have been to use a non-smooth generalisation such as a semismooth Newton method).} For this problem we chose a Papanastasiou-like regularisation (cf.\ \cite{Papanastasiou1987}); the Papanastasiou regularisation has been successfully applied to several problems with Bingham rheology \cite{Chatzimina2005,Dimakopoulos2003,Mitsoulis2004}. The regularised constitutive relation reads:
\begin{equation}\label{papanastasiou2_reg}
	\textcolor{black}{\bm{D} = \frac{1}{2\nu}}\left(\frac{\delta_s (1 - \exp(-M|\bm{S}|))}{|\bm{S}|}+1\right)\bm{S}\quad\text{for }(x-\tfrac{1}{2})^2 + (y-\tfrac{1}{2})^2 \leq (\tfrac{3}{8})^2,
\end{equation}
where $M>0$ is the regularisation parameter (as $M\rightarrow \infty$ we recover the constitutive relation \eqref{Euler_NS_ConstRel}, see \cref{fig:papanastasiou2}); \textcolor{black}{note that this is not related to the regularisation \eqref{mollification}, which has the goal of turning the \emph{measurable} selection into a continuous function}. For the velocity and pressure we used Scott--Vogelius elements and discontinuous piecewise polynomials were used for the stress (cf.\ \cref{exp_ratesC}); 
the problem was implemented in \texttt{firedrake} with $k=1$, $\nu=\frac{1}{2}$, \textcolor{black}{using the same parameters for the linear and nonlinear solvers described in the previous section,} and continuation was employed to reach the values $M=200$ and $\delta_s = 2.5$\textcolor{black}{; more precisely, the problem was initially solved with $M=100$ and $\delta_s=0$ and that solution was used as the Newton guess for the problem with $M+1$ and $\delta_s+0.05$, repeating the procedure until the desired values were reached}. The time step was chosen as $\tau_m = 5\times 10^{-6}$ and the algorithm was applied until \textcolor{black}{the $L^2$ norm of the difference of solutions at subsequent time steps was less than $1\times 10^{-6}$}.
\begin{figure}[!hbtp]
  \centering
\includegraphics[width=.6\textwidth]{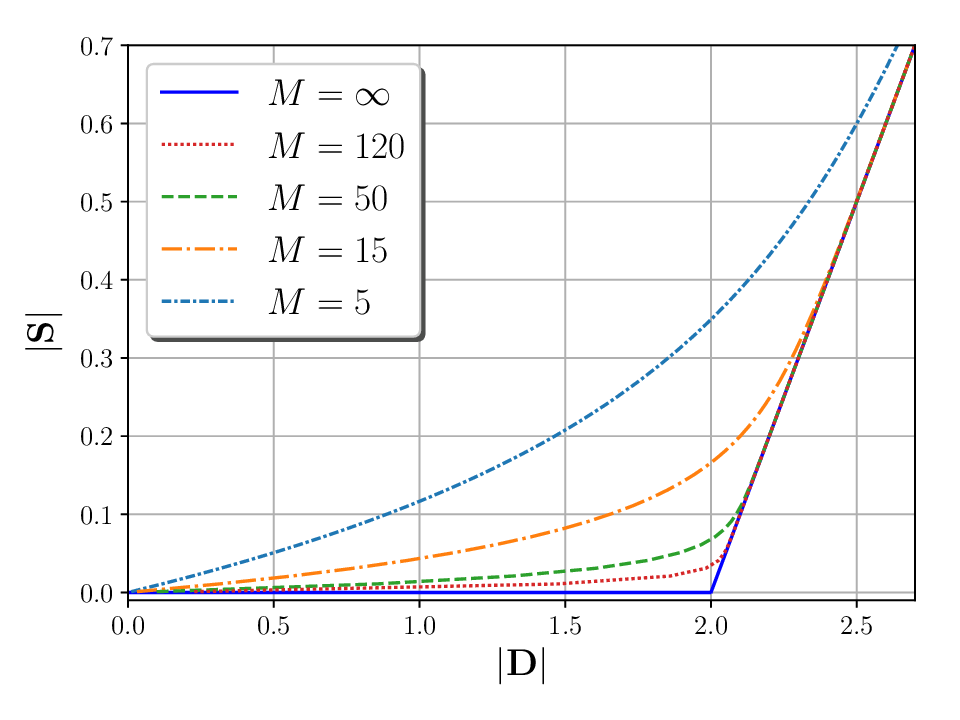}
\vspace{-0.4cm}
  \caption[The short caption]{Regularised constitutive relation for different values of $M$ and $\delta_s = 2$.}
 \label{fig:papanastasiou2}
\end{figure}

Note that when the `yield strain' parameter $\delta_s$ vanishes, we recover the usual Navier--Stokes system. On the other end, if $\delta_s$ is taken to be very large this could be taken as an approximation of the incompressible Euler system in the center of the square; notice how in \cref{fig:streamlines} the fluid picks up more speed in the middle of the domain when $\delta_s>0$ due to the absence of viscosity. This could be an attractive approach to simulating the effects of boundary layers, because it is backed up by a rigorous convergence result; near the boundary the fluid could behave in a Newtonian way and far away $\delta_s$ could be taken arbitrarily large so as to make the effects of the viscosity negligible. This is just one of the possibilities that are yet to be explored within this framework of implicitly constituted fluids and mixed formulations and will be studied in more depth in future work.
\begin{figure}[!htpb]
  \centering
\includegraphics[width=\textwidth]{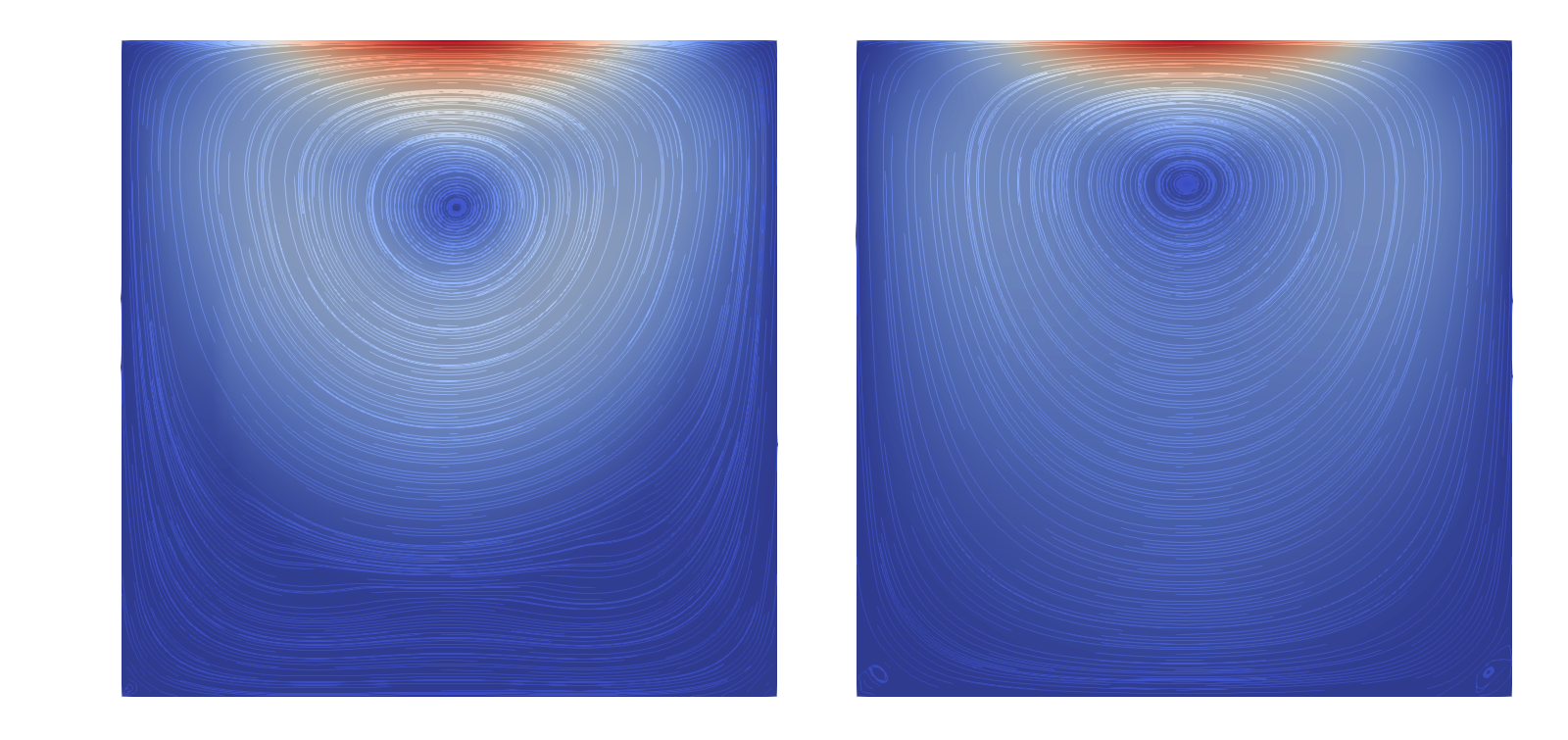}
\vspace{-0.3cm}   
  \caption[The short caption]{Streamlines of the steady state for the problem with $\delta_s = 2.5$ (left) and the Newtonian problem (right).}
  \label{fig:streamlines}
\end{figure}

\Cref{fig:stress_profile} shows the magnitudes of $\bm{S}$ and $\bm{D}$ along the line $x=0.65$ for the steady state of the non-Newtonian problem; it can be clearly seen that the stress is negligibly small for low values of the \textcolor{black}{symmetric velocity gradient} in the center of the square and it then suddenly becomes proportional to it. This transition is not the sharpest in the figure because the regularisation parameter $M$ was not taken sufficiently large, but in the limit this would recover the non-smooth relation. In a sense this is similar to solving a Navier--Stokes problem with high Reynolds number, so for high values of $M$ some stabilisation would be required in order to solve this systems efficiently (even more so if the Newtonian fluid outside of the activation region also has a high Reynolds number); this will be the subject of future research. 
\begin{figure}[!hbtp]
  \centering
\includegraphics[width=.85\textwidth]{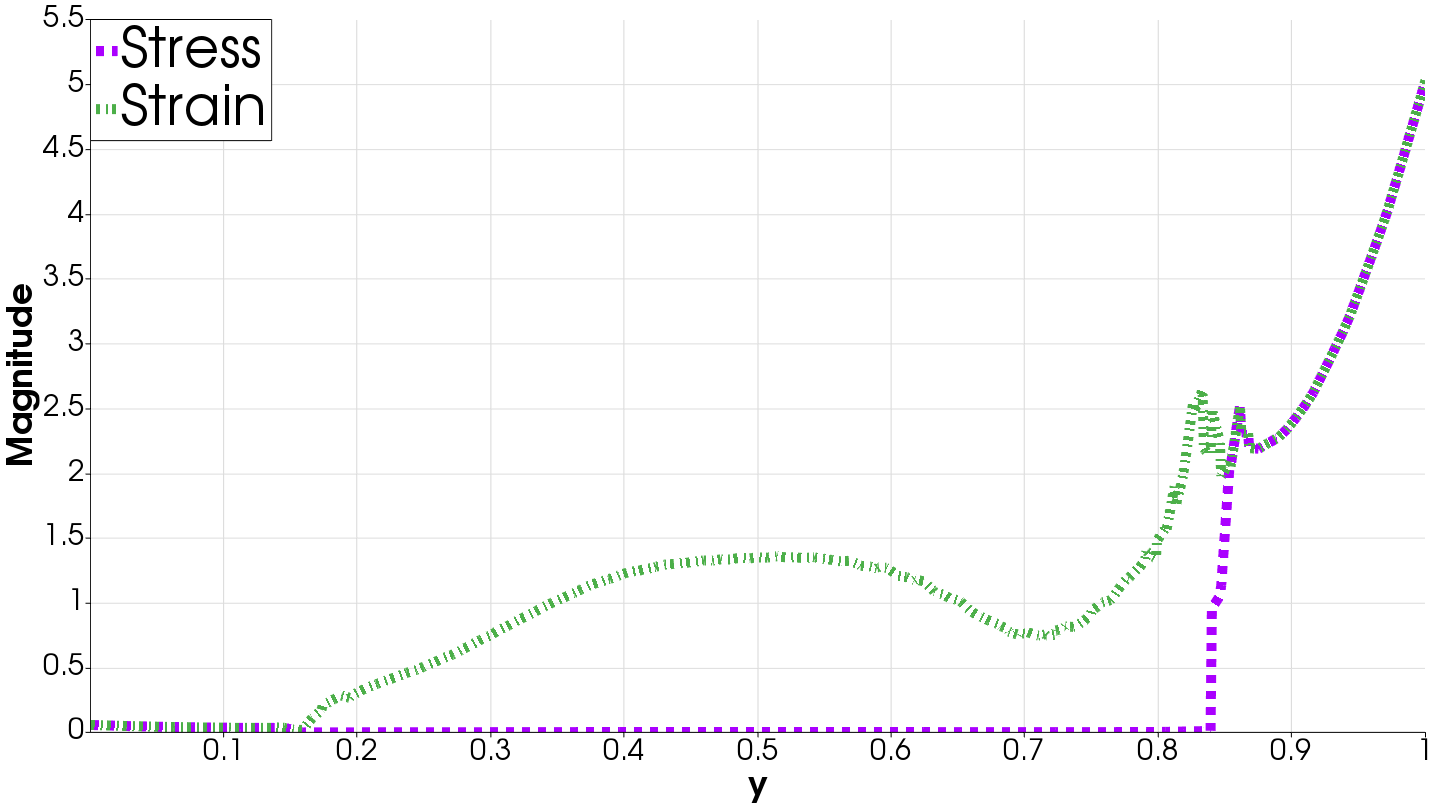}
\vspace{-0.3cm}
  \caption[The short caption]{Magnitude of $\bm{S}$ and $\bm{D}$ at $x=0.65$ for the problem with $\delta_s = 2.5$.}
 \label{fig:stress_profile}
\end{figure}
\subsection{Cessation of the Couette flow of a Bingham fluid}\label{num_bingham}
The flow between two parallel plates induced by the movement at constant speed of one of the plates receives the name of (plane) Couette flow. It is one of the few examples of a configuration that allows us to find an exact solution for the steady Navier--Stokes equations and it is well known that this solution has a linear profile. In this numerical experiment we will take the Couette flow as the initial condition and investigate the behaviour of the system when the plates stop moving. Physically it is expected that the viscosity and no--slip boundary condition will slow down the flow until it finally stops; it can be seen in \cite{Papanastasiou1999} that in the Newtonian case the flow does reach the rest state, albeit in infinite time.

In this section we will solve system \eqref{PDE_goal} with the Bingham constitutive relation:
\begin{equation*}
\renewcommand{\arraystretch}{1.5}
\left\{
\begin{array}{cc}
\bm{S} = \tau_y\frac{\bm{D}}{|\bm{D}|} + 2\nu \bm{D}, & \textrm{ if } |\bm{S}|\geq \tau_y,\\
  \bm{D} = 0, & \textrm{ if }|\bm{S}|<\tau_y,\\
\end{array}
\right.
\end{equation*}
where $\nu>0$ is the viscosity and $\tau_y \geq 0$ is called the yield stress. This is the most common model for a viscoplastic fluid, which is a material that for low stresses (i.e.\ with a magnitude below the yield stress $\tau_y$) behaves like a solid and like a Newtonian fluid otherwise. Interestingly, viscoplastic fluids in the configuration described above reach the rest state in a finite time and there are theoretical upper bounds for the so called \textit{cessation time} (see \cite{Glowinski1984,Huilgol2002}), which makes this a good problem to test the numerical algorithm. \textcolor{black}{Just as in the previous section, for this problem there is also a continuous selection available:}
\begin{equation}
\renewcommand{\arraystretch}{1.5}
\textcolor{black}{\bm{D} = \bm{\mathcal{D}}(\bm{S}):= }\left\{
\begin{array}{cc}
	\textcolor{black}{  \frac{1}{2\nu}(|\bm{S}|-\tau_y)^+\frac{\bm{S}}{|\bm{S}|}, }& \textcolor{black}{\textrm{ if } |\bm{S}|\neq 0,}\\
	\textcolor{black}{ \bm{0}, }&\textcolor{black}{ \textrm{ if }|\bm{S}|=0.}\\
\end{array}
\right.
\end{equation}

For this experiment we \textcolor{black}{again} applied the Papanastasiou regularisation to the non-smooth constitutive relation, \textcolor{black}{ in order to be able to apply Newton's method.} After nondimensionalisation this regularised constitutive law takes the form (compare with \eqref{papanastasiou2_reg}):
\begin{equation}\label{papanastasiou_reg}
\bm{S}(\bm{D}) = \left(\frac{Bn}{|\bm{D}|}(1 - \exp(-M|\bm{D}|))+1\right)\bm{D},
\end{equation}
where $Bn = \frac{\tau_y L}{\nu U}$ is the Bingham number (here $U$ and $L$ are a characteristic velocity and length of the problem, respectively), and $M>0$ is the regularisation parameter (as $M\rightarrow\infty$ we recover the non--smooth relation; compare with \cref{fig:papanastasiou2}).
The problem was solved on the unit square $\Omega = (0,1)^2$ with the following boundary conditions:
\begin{gather*}
\partial\Omega_1 = \{0\}\times(0,1)\cup\{1\}\times(0,1), \qquad \partial\Omega_2 := (0,1)\times\{1\}\cup(0,1)\times\{0\},\\
\bm{u} = \bm{0}\qquad \text{ on }(0,T)\times\partial\Omega_2,\\
\bm{u}_\tau = 0\qquad \text{ on }(0,T)\times\partial\Omega_1,\\
-p + \bm{S}\bm{n}\cdot \bm{n} = 0, \quad \text{ on }(0,T)\times\partial\Omega_1, 
\end{gather*}
where $\bm{u}_\tau$ denotes the component of the velocity tangent to the boundary and $\bm{n}$ is the unit vector normal to the boundary. The initial condition was taken as a standard Couette flow:
\begin{equation*}
\bm{u}(0,\bm{x}) = (1-x_2,0)^\text{T}.
\end{equation*}
For the velocity and pressure we used Taylor--Hood elements and discontinuous piecewise polynomials for the stress\iftoggle{arxiv}{ (cf.\ Section \ref{section_penalty}).}. 
This problem was implemented in FEniCS \cite{Logg2011} \textcolor{black}{using the same parameters for the nonlinear and linear solvers described in the previous section, }with $k=1$ and a timestep $\tau_m$ between $5\times 10^{-7}$ and $1\times 10^{-6}$ for the different values of the Bingham number.
We quantify the change in the flow through the volumetric flow rate (observe that it is constant in $x_1$):
\begin{equation*}
Q(t) := \int_0^1 (1,0)\cdot\bm{u}(t,\bm{x})\,\text{d}x_2,
\end{equation*} 
whose evolution in time is shown in Figure \ref{fig:flow_rate} for different values of the Bingham number. An exponential decay of the flow rate is observed in Figure \ref{fig:flow_rate}, while for positive values of the Bingham number this decay is much faster; these results agree with the ones reported in \cite{Huilgol2002,Chatzimina2005}. In \cite{Chatzimina2005} the problem was solved by integrating a one-dimensional equation for $u_2$; the framework presented here recovers the results obtained there but at the same time has the advantage that it can be applied to a much broader class of problems and geometries.
\begin{figure}[!hbtp]
  \centering
\includegraphics[width=.7\textwidth]{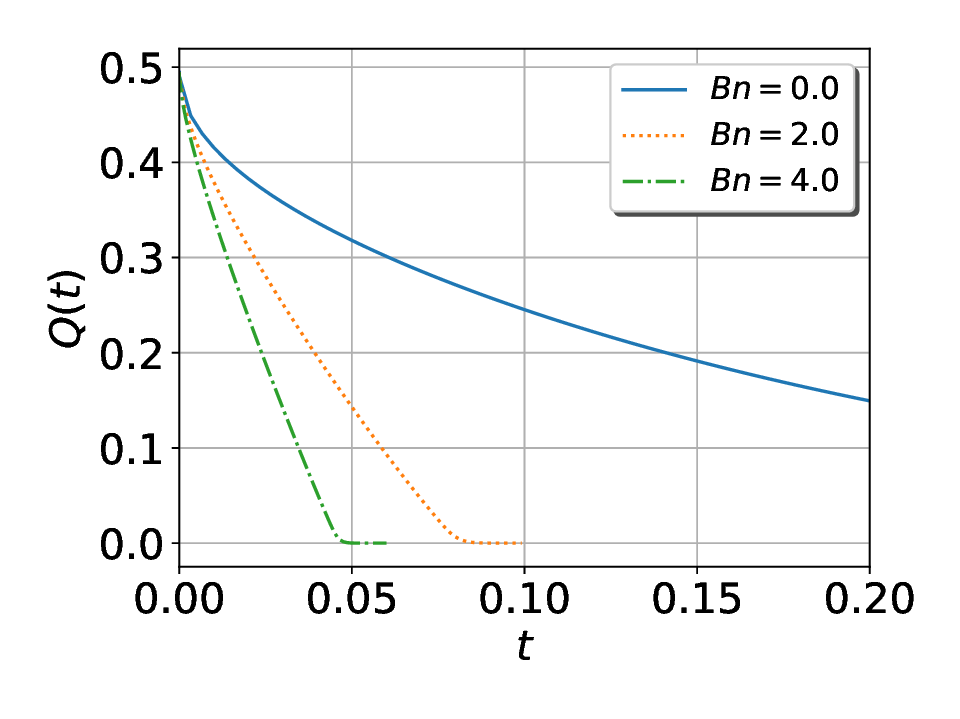}
\vspace{-0.9cm}
  \caption[The short caption]{Evolution of the volumetric flow rate.}
 \label{fig:flow_rate}
\end{figure}

\section{Conclusions}
In this work we presented a 3-field finite element formulation for the numerical approximation of unsteady implicitly constituted incompressible fluids and identified the necessary conditions that guarantee the convergence of the sequence of numerical approximations to a solution of the continuous problem. Although the convergence analysis was written in terms of a selection $\bm{\mathcal{D}}$, the finite element formulation presented here can be used in practice with a fully implicit relation; this is in contrast to the works \cite{Diening:2013,Sueli2018}, where the algorithms relied on finding an approximate constitutive law expressing the stress $\bm{S}^k$ in terms of the \textcolor{black}{symmetric velocity gradient} $\bm{D}^k$, which, while always theoretically possible, is not practical for many models. We also presented numerical experiments that showcase the variety of models that the framework of implicitly constituted models can incorporate.

\subsection{Acknowledgements}
The second author is grateful for useful discussions with J. Málek, T. Tscherpel, and J. Blechta.
\bibliographystyle{siamplain}

\begin{thebibliography}{99}

\bibitem{Arnold1984}
{\sc D.~N. Arnold, F.~Brezzi, and J.~Douglas}, {\em {PEERS: A new mixed finite
  element for plane elasticity}}, Japan J. Appl. Math., 1 (1984), pp.~347--367.

\bibitem{Arnold2002}
{\sc D.~N. Arnold and R.~Winther}, {\em {Mixed finite elements for
  elasticity}}, Numer. Math., 92 (2002), pp.~401--419.

\bibitem{PETSc}
{\sc S.~Balay, S.~Abhyankar, M.~F. Adams, J.~Brown, P.~Brune, K.~Buschelman,
  L.~Dalcin, V.~Eijhout, W.~D. Gropp, D.~Kaushik, M.~G. Knepley, L.~C. McInnes,
  K.~Rupp, S.~Smith, B.~F.~Zampini, H.~Zhang, and H.~Zhang}, {\em Petsc users
  manual}, Tech. Report ANL--95/11--Revision 3.8, Argonne National Laboratory,
  (2017).
\newblock http://www.mcs.anl.gov/petsc.

\bibitem{Behr1993}
{\sc M.~A. Behr, L.~P. Franca, and T.~E. Tezduyar}, {\em {Stabilized finite
  element methods for the velocity-pressure-stress formulation of
  incompressible flows}}, Comput. Methods Appl. Mech. Eng., 104 (1993),
  pp.~31--48.

\bibitem{Belenki:2012}
{\sc L.~Belenki, L.~Berselli, L.~Diening, and M.~Rů{\v{z}}i{\v{c}}ka}, {\em On
  the finite element approximation of p-{S}tokes systems}, SIAM J. Numer.
  Anal., 50 (2012), pp.~373--397, \url{https://doi.org/10.1137/10080436X}.

\bibitem{Berselli2015}
{\sc L.~Berselli, L.~Diening, and M.~Rů{\v{z}}i{\v{c}}ka}, {\em Optimal error
  estimate for semi-implicit space-time discretization for the equations
  describing incompressible generalized {N}ewtonian fluids}, IMA J. Numer.
  Anal., 35 (2015), pp.~680--697.

\bibitem{Blechta2019}
{\sc J.~Blechta, J.~M\'{a}lek, and K.~R. Rajagopal}, {\em {O}n the
  classification of incompressible fluids and a mathematical analysis of the
  equations that govern their motion}, ArXiv Preprint: 1902.04853,  (2019),
  \url{https://arxiv.org/abs/1902.04853}.

\bibitem{Boffi2013}
{\sc D.~Boffi, F.~Brezzi, and M.~Fortin}, {\em Mixed Finite Element Methods and
  Applications}, Springer, 2013.

\bibitem{BREIT2013}
{\sc D.~Breit, L.~Diening, and S.~Schwarzacher}, {\em {S}olenoidal {L}ipschitz
  truncation for parabolic {PDE}s}, Math. Models Methods Appl. Sci., 23 (2013),
  pp.~2671--2700.

\bibitem{Brezzi:1991}
{\sc F.~Brezzi and M.~Fortin}, {\em {M}ixed and {H}ybrid {F}inite {E}lement
  {Methods}}, Springer Ser. Comput. Math., 1991.

\bibitem{PETScLi}
{\sc P.~R. Brune, B.~F. Knepley, B.~F. Smith, and X.~Tu}, {\em Composing
  scalable nonlinear algebraic solvers}, SIAM Rev., 57 (2015), pp.~535--565.

\bibitem{Bulicek:2012aa}
{\sc M.~Bul{\'{i}}{\v{c}}ek, P.~Gwiazda, J.~Malek, K.~R. Rajagopal, and
  A.~{Świerczewska-Gwiazda}}, {\em On flows of fluids described by an implicit
  constitutive equation characterized by a maximal monotone graph}, vol.~402 of
  London Math. Soc. Lecture Note Ser, Cambridge Univ. Press: Cambridge, 2012.

\bibitem{Bulicek:2009}
{\sc M.~Bul\'{i}\v{c}ek, P.~Gwiazda, J.~M\'{a}lek, and
  A.~{\'{S}wierczewska-Gwiazda}}, {\em {O}n steady flows of incompressible
  fluids with implicit power-law-like rheology}, Adv. Calc. Var., 2 (2009),
  pp.~109--136.

\bibitem{Bulicek:2012}
{\sc M.~Bul\'{i}\v{c}ek, P.~Gwiazda, J.~M\'{a}lek, and
  A.~{\'{S}wierczewska-{G}wiazda}}, {\em {O}n unsteady flows of implicitly
  constituted incompressible fluids}, SIAM J. Math. Anal., 44 (2012),
  pp.~2756--2801, \url{https://doi.org/10.1137/110830289}.

\bibitem{Bulicek2018}
{\sc M.~Bulíček, J.~Málek, V.~Průša, and E.~Süli}, {\em {PDE} analysis of
  a class of thermodynamically compatible viscoelastic rate-type fluids with
  stress-diffusion}, in Mathematical Analysis in Fluid Mechanics: Selected
  Recent Results, vol.~710 of AMS Contemporary Mathematics, 2018, pp.~25--51.

\bibitem{Chatzimina2005}
{\sc M.~Chatzimina, G.~C. Georgiou, I.~Argyropaidas, E.~Mitsoulis, and R.~R.
  Huilgol}, {\em Cessation of {C}ouette and {P}oiseuille flows of a {B}ingham
  plastic and finite stopping times}, J. Non-Newtonian Fluid Mech., 129 (2005),
  pp.~117--127.

\bibitem{Clement1975}
{\sc P.~Cl{\'{e}}ment}, {\em Approximation by finite element functions using
  local regularization}, RAIRO, Anal. Num{\'{e}}r., R2,  (1975), pp.~77--84.

\bibitem{Crouzeix1973}
{\sc M.~Crouzeix and P.~A. Raviart}, {\em {Conforming and nonconforming finite
  element methods for solving the stationary Stokes equations I}}, ESAIM: M2AN,
   (1973), pp.~33--75.

\bibitem{Umfpack}
{\sc T.~A. Davis}, {\em Algorithm 832: {UMFPACK V4.3}---an unsymmetric-pattern
  multifrontal method}, ACM Trans. Math. Softw., 30 (2004), pp.~196--199.

\bibitem{DiBenedetto:1993}
{\sc E.~DiBenedetto}, {\em {D}egenerate {P}arabolic {E}quations}, Springer
  Ver., 1993.

\bibitem{Diening:2013}
{\sc L.~Diening, D.~Kreuzer, and E.~S\"{u}li}, {\em {F}inite element
  approximation of steady flows of incompressible fluids with implicit
  power-law-like rheology}, SIAM J. Numer. Anal., 51 (2013), pp.~984--1015,
  \url{https://doi.org/10.1137/120873133}.

\bibitem{Diening2010}
{\sc L.~Diening, M.~Ru{\v{z}}i{\v{c}}ka, and J.~Wolf}, {\em {Existence of weak
  solutions for unsteady motions of generalized Newtonian fluids}}, Ann. Scuola
  Norm. Sup. Pisa Cl. Sci., IX (2010), pp.~1--46.

\bibitem{Diening2017}
{\sc L.~Diening, S.~Schwarzacher, V.~Stroffolini, and A.~Verde}, {\em Parabolic
  {L}ipschitz truncation and caloric approximation}, Calc. Var., 56 (2017).

\bibitem{Dimakopoulos2003}
{\sc Y.~Dimakopoulos and J.~Tsamopoulos}, {\em Transient displacement of a
  viscoplastic material by air in straight and suddenly constricted tubes}, J.
  Non-Newtonian Fluid Mech., 112 (2003), pp.~43--75.

\bibitem{Eckstein2018}
{\sc S.~Eckstein and M.~Rů{\v{z}}i{\v{c}}ka}, {\em On the full space-time
  discretization of the generalized {S}tokes equations: The {D}irichlet case},
  SIAM J. Numer. Anal., 56 (2018), pp.~2234--2261,
  \url{https://doi.org/10.1137/16M1099741}.

\bibitem{Ervin2008}
{\sc V.~J. Ervin, J.~S. Howell, and I.~Stanculescu}, {\em {A dual-mixed
  approximation method for a three-field model of a nonlinear generalized
  Stokes problem}}, Comput. Methods Appl. Mech. Eng., 197 (2008),
  pp.~2886--2900.

\bibitem{Farhloul1993}
{\sc M.~Farhloul and M.~Fortin}, {\em {New mixed finite element for the Stokes
  and elasticity problems}}, SIAM J. Numer. Anal., 30 (1993), pp.~971--990,
  \url{https://doi.org/10.1137/0730051}.

\bibitem{Farhloul1996}
{\sc M.~Farhloul and H.~Manouzi}, {\em {Analysis of non-singular solutions of a
  mixed Navier-Stokes formulation}}, Comput. Methods Appl. Mech. Eng., 129
  (1996), pp.~115--131.

\bibitem{Farhloul2008}
{\sc M.~Farhloul, S.~Nicaise, and L.~Paquet}, {\em {A refined mixed
  finite-element method for the stationary Navier-Stokes equations with mixed
  boundary conditions}}, IMA J. Numer. Anal., 28 (2008), pp.~25--45.

\bibitem{Farhloul2009}
{\sc M.~Farhloul, S.~Nicaise, and L.~Paquet}, {\em {A priori and a posteriori
  error estimations for the dual mixed finite element method of the
  Navier-Stokes problem}}, Numer. Meth. Part. Differ. Equat., 25 (2009),
  pp.~843--869.

\bibitem{Galdi2011}
{\sc G.~P. Galdi}, {\em An Introduction to the Mathematical Theory of the
  Navier-Stokes Equations: Steady State Problems}, Springer, {Second
  Edition}~ed., 2011.

\bibitem{Gmsh}
{\sc C.~Geuzaine and J.~F. Remacle}, {\em Gmsh: a three-dimensional finite
  element mesh generator with built-in pre- and post-processing facilities},
  Int. J. Numer. Meth. Eng., 79 (2009), pp.~1309--1331.

\bibitem{Girault1986}
{\sc V.~Girault and P.~A. Raviart}, {\em Finite Element Methods for
  Navier-Stokes Equations: Theory and Algorithms}, Springer Verlag, 1986.

\bibitem{Girault2003}
{\sc V.~Girault and L.~R. Scott}, {\em {A quasi-local interpolation operator
  preserving the discrete divergence}}, Calcolo,  (2003), pp.~1--19.

\bibitem{Glowinski1984}
{\sc R.~Glowinski}, {\em Numerical Methods for Nonlinear Variational Problems},
  Springer--Verlag, New York, 1984.

\bibitem{Guzman2014a}
{\sc J.~Guzm{\'{a}}n and M.~Neilan}, {\em {Conforming and divergence-free
  Stokes elements in three dimensions}}, IMA J. Numer. Anal.,  (2014),
  pp.~1489--1508.

\bibitem{Guzman2014}
{\sc J.~Guzm{\'{a}}n and M.~Neilan}, {\em {Conforming and divergence-free
  Stokes elements on general triangular meshes}}, Math. Comput., 83 (2014),
  pp.~15--36.

\bibitem{Hirn2013}
{\sc A.~Hirn}, {\em Approximation of the p-{S}tokes equations with equal-order
  finite elements}, J. Math. Fluid Mech, 15 (2013), pp.~65--88.

\bibitem{Howell2009}
{\sc J.~S. Howell}, {\em {Dual-mixed finite element approximation of Stokes and
  nonlinear Stokes problems using trace-free velocity gradients}}, J. Comput.
  Appl. Math, 231 (2009), pp.~780--792.

\bibitem{Howell2016}
{\sc J.~S. Howell and N.~J. Walkington}, {\em Dual-mixed finite element methods
  for the {N}avier-{S}tokes equations}, ESAIM: M2AN, 47 (2013), pp.~789--805.

\bibitem{Hron2017}
{\sc J.~Hron, J.~M\'{a}lek, J.~Stebel, and K.~Tou\v{s}ka}, {\em A novel view on
  computations of steady flows of {B}ingham fluids using implicit constitutive
  relations}, Project MORE Preprint,  (2017).

\bibitem{Huilgol2002}
{\sc R.~R. Huilgol, B.~Mena, and J.~M. Piau}, {\em Finite stopping time
  problems and rheometry of {B}ingham fluids}, J. Non-Newtonian Fluid Mech.,
  102 (2002), pp.~97--107.

\bibitem{refId0}
{\sc C.~Kreuzer and E.~Süli}, {\em Adaptive finite element approximation of
  steady flows of incompressible fluids with implicit power-law-like rheology},
  ESAIM: M2AN, 50 (2016), pp.~1333--1369.

\bibitem{Lions1969}
{\sc J.~L. Lions}, {\em Quelques M\'{e}thodes De R\'{e}solution Des
  Probl\`{e}mes Aux Limites Non Lin\'{e}aires}, Dunod, Paris, 1969.

\bibitem{Logg2011}
{\sc A.~Logg, K.~A. Mardal, and G.~N. Wells}, {\em {FEniCS : Automated Solution
  of Differential Equations by the Finite Element Method}}, Springer, 2011,
  \url{https://doi.org/10.1007/978-3-642-23099-8}.

\bibitem{Maringova2018}
{\sc E.~Maringová and J.~Žabenský}, {\em On a
  {N}avier–{S}tokes–{F}ourier-like system capturing transitions between
  viscous and inviscid fluid regimes and between no-slip and perfect-slip
  boundary conditions}, Nonlinear Anal. Real World Appl., 41 (2018),
  pp.~152--178.

\bibitem{Mitsoulis2004}
{\sc E.~Mitsoulis and R.~R. Huilgol}, {\em Entry flows of {B}ingham plastics in
  expansions}, J. Non-Newtonian Fluid Mech., 122 (2004), pp.~45--54.

\bibitem{Papanastasiou1987}
{\sc T.~C. Papanastasiou}, {\em Flows of materials with yield}, J. Rheol, 31
  (1987), pp.~385--404.

\bibitem{Papanastasiou1999}
{\sc T.~C. Papanastasiou, G.~Georgiou, and A.~Alexandrou}, {\em Viscous Fluid
  Flow}, CRC Press, Boca Raton, 1999.

\bibitem{Prusa2018}
{\sc V.~Průša and K.~R. Rajagopal}, {\em A new class of models to describe
  the response of electrorheological and other field dependent fluids}, Adv.
  Struct. Mater., 89 (2018), pp.~655--673.

\bibitem{Rajagopal:2003}
{\sc K.~R. Rajagopal}, {\em On implicit constitutive theories}, Appl. Math., 48
  (2003), pp.~279--319.

\bibitem{Rajagopal2006}
{\sc K.~R. Rajagopal}, {\em {On implicit constitutive theories for fluids}}, J.
  Fluid Mech., 550 (2006), pp.~243--249.

\bibitem{Rajagopal2007}
{\sc K.~R. Rajagopal}, {\em {The elasticity of elasticity}}, Z. Angew. Math.
  Phys, 5807 (2007), pp.~309--317.

\bibitem{Rajagopal2008}
{\sc K.~R. Rajagopal and A.~R. Srinivasa}, {\em {On the thermodynamics of
  fluids defined by implicit constitutive relations}}, Z. Angew. Math. Phys, 59
  (2008), pp.~715--729.

\bibitem{Rathgeber2016}
{\sc F.~Rathgeber, D.~A. Ham, L.~Mitchell, M.~Lange, F.~Luporini, A.~T.~T.
  Mcrae, G.-T. Bercea, G.~R. Markall, and P.~H.~J. Kelly}, {\em Firedrake:
  automating the finite element method by composing abstractions}, ACM Trans.
  Math. Softw., 43 (2016).

\bibitem{Roubicek2013}
{\sc T.~Roubi{\v{c}}ek}, {\em Nonlinear Partial Differential Equations with
  Applications}, Birkhäuser, second edition~ed., 2013.

\bibitem{Ruas1994}
{\sc V.~Ruas}, {\em {An optimal three-field finite element approximation of the
  Stokes system with continuous extra stresses}}, Japan J. Appl. Math.,
  (1994), pp.~113--130.

\bibitem{Sandri1998}
{\sc D.~Sandri}, {\em {A posteriori estimators for mixed finite element
  approximations of a fluid obeying the power law}}, Comput. Methods Appl.
  Mech. Eng.,  (1998), pp.~329--340.

\bibitem{Scott1985}
{\sc L.~R. Scott and M.~Vogelius}, {\em {Norm estimates for a maximal right
  inverse of the divergence operator in spaces of piecewise polynomials}},
  Math. Modelling Numer. Anal., 19 (1985), pp.~111--143.

\bibitem{Simon1987}
{\sc J.~Simon}, {\em Compact sets in the space {$L^p(0,T;B)$}}, Ann. Mat. Pura
  Appl., 4 (1987), pp.~65--96.

\bibitem{Sueli2018}
{\sc E.~Süli and T.~Tscherpel}, {\em Fully discrete finite element
  approximation of unsteady flows of implicitly constituted incompressible
  fluids}, IMA J. Numer. Anal.,  (2018), pp.~1--49.

\bibitem{Tscherpel2018}
{\sc T.~Tscherpel}, {\em {FEM for the Unsteady Flow of Implicitly Constituted
  Incompressible Fluids}}, PhD thesis, University of Oxford, 2018.	
\end{thebibliography}

\end{document}